\documentclass[11pt]{article}
\usepackage{etex}
\usepackage{mathtools}
\usepackage{enumitem}

\usepackage{amsmath,amsthm}
\usepackage{csquotes}
\usepackage{todonotes}

\usepackage{amssymb}

\usepackage{amssymb}

\usepackage{graphicx}

\usepackage{subfig}
\usepackage[final]{showkeys}

\usepackage{etoolbox}
\usepackage{relsize}

\usepackage{wasysym}

\usepackage{mathrsfs}

\usepackage{mathpazo}

\usepackage[titletoc]{appendix}
\usepackage[doc]{optional}

\usepackage{soul}

\usepackage{colortbl,booktabs,sectsty,multirow}
\usepackage{xcolor}

\usepackage{cancel}

\usepackage{empheq}

\definecolor{myblue}{rgb}{.8, .8, 1}
  \newcommand*\mybluebox[1]{
    \colorbox{myblue}{\hspace{1em}#1\hspace{1em}}}

\usepackage[obeyspaces,hyphens,spaces]{url}

\usepackage{hyperref}
\hypersetup{
    colorlinks=true,
    linkcolor=blue,
    citecolor=blue,
    filecolor=magenta,
    urlcolor=cyan
}

\usepackage[
  open,
  openlevel=2,
  atend,
  numbered
]{bookmark}

\usepackage[capitalize, nameinlink, noabbrev]{cleveref}
\crefname{equation}{}{}
\crefname{chapter}{Chapter}{Chapters}
\crefname{item}{item}{items}
\crefname{figure}{Figure}{Figures}
\crefname{theorem}{Theorem}{Theorems}
\crefname{lemma}{Lemma}{Lemmas}
\crefname{proposition}{Proposition}{Propositions}
\crefname{corollary}{Corollary}{Corollarys}
\crefname{definition}{Definition}{Definitions}
\crefname{fact}{Fact}{Facts}
\crefname{example}{Example}{Examples}
\crefname{algorithm}{Algorithm}{Algorithms}
\crefname{remark}{Remark}{Remarks}
\crefname{note}{Note}{Notes}
\crefname{notation}{Notation}{Notations}
\crefname{case}{Case}{Cases}
\crefname{exercise}{Exercise}{Exercises}
\crefname{question}{Question}{Questions}
\crefname{claim}{Claim}{Claims}
\crefname{enumi}{}{}

\usepackage[top= 2cm, bottom = 2 cm, left = 2.2 cm, right= 2.2 cm]{geometry}

\usepackage{float}

\usepackage{pgf}

\usepackage{array}
\usepackage{tabu}

\allowdisplaybreaks

\numberwithin{equation}{section}

\theoremstyle{plain}
\newtheorem{theorem}{Theorem}[section]

\newtheorem{corollary}[theorem]{Corollary}
\newtheorem{fact}[theorem]{Fact}
\newtheorem{lemma}[theorem]{Lemma}

\theoremstyle{definition}
\newtheorem{definition}[theorem]{Definition}
\newtheorem{example}[theorem]{Example}

\newtheorem{remark}[theorem]{Remark}

\newcommand{\argmin}{\ensuremath{\operatorname{argmin}}}
\newcommand{\inte}{\ensuremath{\operatorname{int}}}
\newcommand{\bd}{\ensuremath{\operatorname{bd}}}

\newcommand{\aff}{\ensuremath{\operatorname{aff} \,}}

\newcommand{\weakly}{\ensuremath{{\;\operatorname{\rightharpoonup}\;}}}

\newcommand{\dom}{\ensuremath{\operatorname{dom}}}

\newcommand{\Fix}{\ensuremath{\operatorname{Fix}}}
\newcommand{\Id}{\ensuremath{\operatorname{Id}}}

\newcommand{\C}{\ensuremath{\operatorname{C}}}
\newcommand{\D}{\ensuremath{\operatorname{D}}}

\newcommand{\Pro}{\ensuremath{\operatorname{P}}}

\newcommand{\I}{\ensuremath{\operatorname{I}}}
\newcommand{\J}{\ensuremath{\operatorname{J}}}

\newcommand{\CCO}[1]{CC{#1}}

\providecommand{\norm}[1]{\left\lVert#1\right\rVert}
\providecommand{\innp}[1]{\left\langle#1\right\rangle}

\providecommand{\Innp}[1]{\Big\langle#1\Big\rangle}

\begin{document}

\title{ 
	Bregman Circumcenters: Monotonicity and Forward Weak Convergence
}

\author{
         Hui\ Ouyang\thanks{
                 Mathematics, University of British Columbia, Kelowna, B.C.\ V1V~1V7, Canada.
                 E-mail: \href{mailto:hui.ouyang@alumni.ubc.ca}{\texttt{hui.ouyang@alumni.ubc.ca}}.}
                 }

\date{March 26, 2022}
\maketitle

\begin{abstract}
	\noindent
	Recently, we systematically studied the basic theory of Bregman circumcenters in another paper.  In this work, we aim to apply Bregman circumcenters to optimization algorithms.
	 
	Here, we propose the forward Bregman monotonicity which is a generalization of the powerful Fej\'er monotonicity  and show a weak convergence result of the forward Bregman monotone sequence. We also naturally introduce the Bregman  circumcenter mappings associated with a finite set of operators. Then we provide sufficient conditions for the sequence of iterations of the forward Bregman circumcenter mapping to be forward Bregman monotone.
	Furthermore,   we prove that  the sequence of iterations of the forward Bregman circumcenter mapping weakly converges to a point in the intersection of the fixed point sets of relevant operators, which reduces to the known weak convergence result of the circumcentered method under the Euclidean distance. In addition,  particular examples are provided to illustrate the Bregman isometry and Browder's demiclosedness principle, and our convergence result. 
\end{abstract}
	
	{\small
		\noindent
		{\bfseries 2020 Mathematics Subject Classification:}
		{
			Primary 90C48, 49M37, 47H05;  
			Secondary 90C25,  65K05, 52A41.
		}
		
		\noindent{\bfseries Keywords:}
		Bregman distance, Legendre function, 
 backward Bregman (pseudo)-circumcenter, 
  forward Bregman (pseudo)-circumcenter, 
		  fixed point set, forward Bregman   monotonicity, convergence, Bregman circumcenter method.
	}
\section{Introduction}
Throughout the work, we assume that $\mathbb{N}=\{0,1,2,\ldots\}$ and $\{ m,  n \} \subseteq  \mathbb{N} \smallsetminus \{0\}$, and that 
\begin{empheq}[box=\mybluebox]{equation*}
\mathcal{H}    \text{ is a real Hilbert space  with inner product } \innp{\cdot, \cdot} \text{ and induced norm } \norm{\cdot}.	
\end{empheq}

Projection methods, including the method of cyclic projections based on  Bregman distances, are employed in many applications (see, e.g., \cite{BBC2001}, \cite{BBC2003}, \cite{BC2003}, and \cite{Bregman1967} and the references therein). The convergence rate of certain circumcentered methods is no worse than  that of the method of cyclic projections under the Euclidean distance, and, in some cases, some circumcentered methods converge much faster than the method of cyclic projections, for solving the best approximation problem   or the feasibility problem (see, \cite{AABBIS2021}, \cite{BBCS2020CRMbetter}, \cite{BOyW2019Isometry}, \cite{BOyW2019LinearConvergence}, \cite{BOyW2020BAM}, \cite{BBCS2017}, \cite{BBCS2018}, \cite{BBCS2019}, \cite{BBCS2020ConvexFeasibility},   and \cite{Ouyang2020Finite} for details). 

In our recent work \cite{BOyW2021BregmanCircumBasicTheory}, we presented multiple interesting theoretical results on the Bregman circumcenter.  Various examples under general Bregman distances were also provided to illustrate our main results. 

In this work, our objects are to introduce the circumcenter mappings and methods and investigate the convergence of circumcenter methods under general Bregman distances.

We present the main results in this work below. 
\begin{itemize}
	\item[\textbf{R1:}]  \cref{theorem:ConvBregBackwMonot} characterizes the weak convergence of the  forward Bregman monotone sequence.

	\item[\textbf{R2:}]   \Cref{theorem:Convergence,theorem:xkConvergence} show that  the sequence of iterations of the forward Bregman circumcenter mapping  convergences weakly to a point in the intersection of the fixed point sets of   operators   inducing the mapping. 
\end{itemize}

The remainder of the work is organized as follows. 
In  \cref{sec:Preliminaries}, we collect fundamental  definitions and facts. In \cref{sec:Monotonicity}, we introduce the forward Bregman monotonicity and investigate the weak convergence of the  forward Bregman monotone sequence.
In \cref{sec:Bregmancircumcentermethods}, we introduce the forward Bregman circumcenter mapping  induced by a finite set of operators 
and   specify sufficient conditions for the weak convergence of the forward Bregman circumcenter method to a point in the intersection of fixed point sets of the related operators. Moreover, we provide particular examples to illuminate our hypotheses and main results under general Bregman distances. 

We now turn to the notation used in this paper.  $\Gamma_{0} (\mathcal{H}) $ is the set of proper closed convex functions from $\mathcal{H}$ to $\left]-\infty, +\infty\right]$.
Let $f:\mathcal{H} \to \left]-\infty, +\infty\right]$ be proper.  The \emph{domain}  (\emph{conjugate function,   gradient, subgradient},  respectively) of $f$ is denoted by $\dom f$ ($f^{*}$,    $\nabla f$, $\partial f$,   respectively).   We say $f$ is \emph{coercive} if $\lim_{\norm{x} \to +\infty} f(x) = +\infty$.
 Let $C$ be a nonempty subset of $\mathcal{H}$.  Its \emph{interior} and \emph{boundary} are abbreviated by $\inte C$  and $\bd C$, respectively.  
 $C$ is an \emph{affine subspace} of
 $\mathcal{H}$ if $C \neq \varnothing$ and $(\forall \rho\in\mathbb{R})$ $\rho
 C + (1-\rho)C = C$. The smallest affine subspace of $\mathcal{H}$ containing $C$ is
  denoted by $\aff C$ and called the \emph{affine hull} of $C$. 
 The \emph{best approximation operator} (or \emph{projector}) onto $C$ under the Euclidean distance is denoted by $\Pro_{C}$, that is, $(\forall x \in \mathcal{H} )$ $\Pro_{C}x :=  \argmin_{y \in C} \norm{x-y}$. $\iota_{C}$ is the \emph{indicator function of $C$}, that is, $(\forall x \in C)$ $\iota_{C} (x) =0$ and $(\forall x \in \mathcal{H} \smallsetminus C)$ $\iota_{C} (x) =+ \infty$. 
 $\Id$ stands for the \emph{identity mapping}. 
 For every $x \in \mathcal{H}$ and $\delta \in \mathbb{R}_{++}$,  $B [x; \delta] $ is the \emph{closed ball with center at $x$ and with radius $\delta$}. Let $A : \mathcal{H} \to 2^{\mathcal{H}}$ and let $x \in \mathcal{H}$. Then $A$ is \emph{locally bounded at $x$} if there exists $\delta \in \mathbb{R}_{++}$ such that $A(B[x;\delta] )$ is bounded.  Denote by $\Fix A := \{x \in \mathcal{H} ~:~ x \in A(x) \}$.
For other notation not explicitly defined here, we refer the reader to \cite{BC2017}.

\section{Bregman distances and projections} \label{sec:Preliminaries}
In this section, we  collect some essential definitions and facts to be used subsequently.

It is clear that if $f = \frac{1}{2} \norm{\cdot}^{2}$ in the following  \cref{defn:BregmanDistance}, we recover the  Euclidean distance $\D :   \mathcal{H} \times \mathcal{H} \to \left[0, +\infty\right] :  (x,y) \mapsto \frac{1}{2} \norm{x-y}^{2}$.
\begin{definition} {\rm \cite[Definitions~7.1 and 7.7]{BBC2001}} \label{defn:BregmanDistance}
	Suppose that $f \in \Gamma_{0} (\mathcal{H}) $  with $\inte \dom f \neq \varnothing$ and that $f$ is G\^ateaux differentiable  on $\inte \dom f$. The  \emph{Bregman distance $\D_{f}$ associated with $f$} is defined by
	\begin{align*}
\D_{f}: \mathcal{H} \times \mathcal{H} \to \left[0, +\infty\right] : (x,y) \mapsto \begin{cases}
f(x) -f(y) -\innp{\nabla f(y), x-y}, \quad &\text{if } y \in  \inte \dom f;\\
+\infty,  \quad &\text{otherwise}.
\end{cases} 
\end{align*}
	Moreover, let $C$ be a nonempty subset of $\mathcal{H}$. For every $(x,y) \in \dom f \times \inte \dom f$, define the \emph{backward Bregman projection} (or simply  \emph{Bregman projection}) of $y$ onto $C$ and  \emph{forward  Bregman projection} of $x$ onto $C$,  respectively, as 
\begin{align*}
&\overleftarrow{\Pro}^{f}_{C}(y):= \left\{ u \in C \cap \dom f ~:~  (\forall c \in \C) \D_{f} \left( u,y \right) \leq \D_{f} (c,y)   \right\}, \text{and}\\
&\overrightarrow{\Pro}^{f}_{C}(x) := \left\{  v \in C \cap \inte \dom f ~:~ (\forall c \in \C) \D_{f} \left( x,v \right) \leq \D_{f} (x, c)  \right\}.
\end{align*} 
Abusing notation slightly, we shall write 	$\overleftarrow{\Pro}^{f}_{C}(y) =u$ and $\overrightarrow{\Pro}^{f}_{C}(x)=v$, if $\overleftarrow{\Pro}^{f}_{C}(y)$ and $\overrightarrow{\Pro}^{f}_{C}(x)$ happen to be the singletons $\overleftarrow{\Pro}^{f}_{C}(y) =\{ u\}$  and $\overrightarrow{\Pro}^{f}_{C}(x) =\{v\} $, respectively.  
\end{definition}

The following definitions are compatible with their classical counterparts as  \cite[Definitions~2.1, 2.3 and 2.8]{BB1997Legendre} in the finite-dimensional Euclidean space (see, \cite[Theorem~5.11]{BBC2001} for more details). 
\begin{definition} {\rm \cite[Definition~5.2 and Theorem~5.6]{BBC2001}} \label{def:Legendre}
	Suppose  that $f \in \Gamma_{0} (\mathcal{H}) $.  We say $f$ is:
	\begin{enumerate}
		\item  \emph{essentially smooth}, if
$\dom \partial f =\inte \dom f \neq \varnothing$, $f$ is G\^ateaux differentiable on $\inte \dom f$, and $\norm{\nabla f (x_{k})} \to +\infty$, for every sequence $(x_{k})_{k \in \mathbb{N}}$ in $\inte \dom f$ converging to some point in $\bd \dom f$.
		\item  \label{def:Legendre:convex} \emph{essentially strictly convex}, if $(\partial f)^{-1}$ is locally bounded on its domain and $f$ is strictly convex on every convex subset of $\dom \partial f$.
		\item \emph{Legendre} (or a \emph{Legendre function} or a \emph{convex function of Legendre type}),  if $f$ is both essentially smooth and essentially strictly convex.
	\end{enumerate}
\end{definition}

\begin{fact}{\rm \cite[Lemma~7.3]{BBC2001}}  \label{fact:PropertD}
	Suppose that $f \in \Gamma_{0} (\mathcal{H}) $  with $\inte \dom f \neq \varnothing$, that $f$ is G\^ateaux differentiable  on $\inte \dom f$, and that $f$ is essentially strictly convex. Let $x \in \mathcal{H}$ and $y \in \inte \dom f$. Then $D_{f}(\cdot, y)$ is coercive. Moreover, $D_{f}(x, y) =0 \Leftrightarrow x =y$.
\end{fact}

 The following \cref{fact:Pleft,fact:charac:PrightCf}   on the existence and uniqueness of backward and forward Bregman projections
  are fundamental to some results in \cref{sec:Bregmancircumcentermethods} below. 
\begin{fact} {\rm \cite[Corollary~7.9]{BBC2001}} \label{fact:Pleft}
	Suppose  that $f \in \Gamma_{0} (\mathcal{H}) $   is Legendre, that $C$ is a closed convex subset of $\mathcal{H}$ with $C \cap \inte \dom f \neq \varnothing$, and that $y \in \inte \dom f $. Then 
	$\overleftarrow{\Pro}^{f}_{C}(y) $  is a singleton contained in $C \cap \inte \dom f$.
\end{fact}

Not surprisingly, not every Legendre function allows forward Bregman projections, and  the backward and forward Bregman projections are different notions (see, e.g.,  \cite{BB1997Legendre}, \cite{BC2003} and \cite{BD2002} for details). In particular, we can find well-known functions satisfying the hypotheses of the following \cref{fact:charac:PrightCf}    in \cite[Examples~2.1 and 2.7]{BC2003}. We refer the interested readers to \cite{BB1997Legendre}, \cite{BC2003} and \cite{BD2002} for details on the backward and forward Bregman projections.

\begin{definition} {\rm \cite[Definition~2.4]{BC2003}}
	Suppose that   $f \in \Gamma_{0} (\mathcal{H}) $  is  Legendre  such that $\dom f^{*}$ is open.  We say \emph{the function $f$ allows forward Bregman projections} if it satisfies the following properties.
	\begin{enumerate}
		\item $\nabla^{2} f$ exists and is continuous on $\inte \dom f$.
		\item $\D_{f}$ is convex on $\inte \dom f \times \inte \dom f$.
		\item For every $x \in \inte \dom f$, $D_{f}(x, \cdot)$ is strictly convex on $\inte \dom f$.
	\end{enumerate}
\end{definition}

\begin{fact} {\rm \cite[Fact~2.6]{BC2003}}  \label{fact:charac:PrightCf}
	Suppose that $\mathcal{H} =\mathbb{R}^{n}$ and  $f \in \Gamma_{0} (\mathcal{H}) $ is  Legendre  such that $\dom f^{*}$ is open, that $f$ allows forward Bregman projections, and that  $C $ is a closed convex subset of $\mathcal{H}$ with $C \cap \inte \dom f \neq \varnothing$. Then $(\forall x \in \inte \dom f)$ $\overrightarrow{\Pro}^{f}_{C}(x)  $ is a singleton contained in $C \cap \inte \dom f$.
\end{fact}

\section{Forward Bregman monotonicity} \label{sec:Monotonicity}
In this section, we show the weak convergence of the forward Bregman monotone sequence, which plays a critical role in the proof of our main result in the next section. 

The following is a variant of the Bregman monotonicity defined in \cite[Definition~1.2]{BBC2003}. Note that both     \cite[Definition~1.2]{BBC2003} and the following \cref{defn:BregBackMonotone} are natural generalizations of the classical Fej\'er monotonicity.  In view of the broad applications of Fej\'er monotonicity and Bregman monotonicity in the proof of the convergence of iterative algorithms (see, e.g., \cite{BC2017}, \cite{BBC2003}, \cite{Cegielski} and the references therein), we assume the forward Bregman monotonicity defined in \cref{defn:BregBackMonotone} below is interesting on its own and probably can be used in many other iterative algorithms under general Bregman distances.  
\begin{definition} \label{defn:BregBackMonotone}
	A sequence $(x_{k})_{k \in \mathbb{N}}$ in $\mathcal{H}$ is  \emph{forward Bregman   monotone with respect to a set $C \subseteq \mathcal{H}$} if $C\cap  \inte \dom f \neq \varnothing$,  $(x_{k})_{k \in \mathbb{N}}$  lies in $\inte \dom f $, and 
	\begin{align*}
	(\forall c \in  C \cap \inte \dom f )~ (\forall k \in \mathbb{N}) \quad \D_{f} (x_{k+1}, c ) \leq \D_{f} (x_{k} , c).
	\end{align*}
\end{definition}

\begin{fact} \label{fact:strictlyConvex}
	Suppose that  $f \in \Gamma_{0} (\mathcal{H}) $  with $\inte \dom f \neq \varnothing$, that $f$ is G\^ateaux differentiable  on $\inte \dom f$, and that $f$ is essentially strictly convex. Then $f$ is strictly convex on $\inte \dom f$, which is equivalent to  
	\begin{align*}
	(\forall x \in \inte \dom f ) (\forall y \in \inte \dom f )  \quad x \neq y \Rightarrow \innp{x-y, \nabla f(x) - \nabla f(y)} >0.
	\end{align*}
\end{fact}

\begin{proof}
	Because $f$ is convex, by \cite[Propositions~3.45(ii), 8.2 and 16.27]{BC2017}, $\inte \dom f$ is a  convex subset of $\dom \partial f$. Then employ  \cref{def:Legendre}\cref{def:Legendre:convex} to see that $f$ is strictly convex on $\inte \dom f$.
Hence, the last equivalence is obtained by applying \cite[Proposition~17.10]{BC2017} to the function $f+\iota_{\inte \dom f}$.
\end{proof}

\begin{theorem} \label{theorem:ConvBregBackwMonot}
	Suppose that  $f \in \Gamma_{0} (\mathcal{H}) $  with $\inte \dom f \neq \varnothing$ and that $f$ is G\^ateaux differentiable  on $\inte \dom f$. Let $C $ be a closed convex set in $\mathcal{H}$ with $C\cap  \inte \dom f \neq \varnothing$, and let  $(x_{k})_{k \in \mathbb{N}}$ be in $ \inte \dom f$  and be forward Bregman  monotone with respect to $C$.
	Then the following statements hold.
	\begin{enumerate}
		\item  \label{theorem:ConvBregBackwMonot:D} $\left(\forall c \in C\cap  \inte \dom f  \right)$  $\left( \D_{f} (x_{k} , c) \right)_{k \in \mathbb{N}}$ is decreasing, nonnegative and convergent.
			\item  \label{theorem:ConvBregBackwMonot:limexist} Let $\{ y,z\} \subseteq C\cap  \inte \dom f$. Then the limit $\lim_{k \to \infty} \innp{\nabla f(y) - \nabla f(z) , x_{k}}$ exists. 
		\item \label{theorem:ConvBregBackwMonot:x} Suppose that $f$ is essentially strictly convex. Then $(x_{k} )_{k \in \mathbb{N}}$ is bounded. 
		\item \label{theorem:ConvBregBackwMonot:conv} Suppose that  $f$ is essentially strictly convex. Then $(x_{k})_{k \in \mathbb{N}}$ weakly converges   to some point in $ C\cap  \inte \dom f$ if and only if all  weak sequential cluster points of $(x_{k})_{k \in \mathbb{N}}$ are in $ C\cap  \inte \dom f$.
	\end{enumerate}
\end{theorem}

\begin{proof}
	\cref{theorem:ConvBregBackwMonot:D}: This is clear from   \cref{defn:BregBackMonotone,defn:BregmanDistance}. 
	
	\cref{theorem:ConvBregBackwMonot:limexist}: According to \cref{defn:BregmanDistance}, for every $k \in \mathbb{N}$,
	\begin{align*}
	 &\D_{f} (x_{k} , y) -\D_{f} (x_{k} ,z) = f(x_{k}) - f(y) -\innp{ \nabla f (y) , x_{k} -y} - \Big(   f(x_{k}) - f(z) -\innp{ \nabla f (z) , x_{k} -z} \Big)\\ 
\Leftrightarrow &	\innp{ \nabla f (z) -\nabla f (y) , x_{k}}=\D_{f} (x_{k} , y) -\D_{f} (x_{k} ,z) +f(y)-f(z) -\innp{ \nabla f (y) ,  y} +\innp{ \nabla f (z) , z}.
	\end{align*}
	 which, via \cref{theorem:ConvBregBackwMonot:D}, yields the desired result.
	 
	\cref{theorem:ConvBregBackwMonot:x}: Let $c \in C\cap  \inte \dom f$. 
	Then based on \cref{fact:PropertD},  the boundedness of	$\left( \D_{f} (x_{k} , c) \right)_{k \in \mathbb{N}}$   implies the boundedness of $(x_{k} )_{k \in \mathbb{N}}$. 

\cref{theorem:ConvBregBackwMonot:conv}:	\enquote{$\Rightarrow$}: This is trivial.
	
	\enquote{$\Leftarrow$}: Suppose that all weak sequential cluster points of $(x_{k})_{k \in \mathbb{N}}$ are in $C\cap  \inte \dom f$.    
Because the boundedness of $(x_{k} )_{k \in \mathbb{N}}$ is proved in \cref{theorem:ConvBregBackwMonot:x} above, bearing \cite[Lemma~2.46]{BC2017} in mind, we know that it remains to prove that $(x_{k} )_{k \in \mathbb{N}}$ has at most one weak sequential cluster point in $ C \cap \inte \dom f$. Assume that $ \bar{x} $ and $\hat{x}$ are two weak sequential cluster points of $(x_{k} )_{k \in \mathbb{N}}$ in $C\cap  \inte \dom f$. Then there exist subsequences $( x_{k_{j}} )_{j \in \mathbb{N}} $ and $( x_{k_{t}} )_{t \in \mathbb{N}} $ of $ ( x_{k})_{k \in \mathbb{N}}$ such that $x_{k_{j}} \weakly \bar{x}$ and $x_{k_{t}} \weakly \hat{x}$.
Then
\begin{subequations}\label{eq:prop:ConvBregBackwMonot:conv}
	\begin{align}
	& \innp{\nabla f(\bar{x}) - \nabla f(\hat{x}) , x_{k_{j}}} \to  \innp{\nabla f(\bar{x}) - \nabla f(\hat{x}), \bar{x}}; \\
	 & \innp{\nabla f(\bar{x}) - \nabla f(\hat{x}) , x_{k_{t}}} \to  \innp{\nabla f(\bar{x}) - \nabla f(\hat{x}) , \hat{x}}. 
	\end{align}
\end{subequations}
On the other hand, apply \cref{theorem:ConvBregBackwMonot:limexist} with $y=\bar{x}$ and $z=\hat{x}$ to deduce that $\lim_{k \to \infty} \innp{\nabla f(\bar{x}) - \nabla f(\hat{x}) , x_{k}}$  exists. 
 This combined with \cref{eq:prop:ConvBregBackwMonot:conv} implies that $\innp{\nabla f(\bar{x}) - \nabla f(\hat{x}), \bar{x}} = \innp{\nabla f(\bar{x}) - \nabla f(\hat{x}), \hat{x}} $, that is, 
 \begin{align*} 
 \innp{\nabla f(\bar{x}) - \nabla f(\hat{x}), \bar{x}- \hat{x}} =0,
 \end{align*}
 which, combining with \cref{fact:strictlyConvex}, yields $\hat{x} =\bar{x}$. Altogether, the proof is complete. 
\end{proof}

\section{Bregman circumcenter methods} \label{sec:Bregmancircumcentermethods}
As we mentioned in the introduction, projection methods based on  Bregman distances have broad applications and some circumcentered methods accelerate the method of cyclic projections  for finding the best approximation point onto, or  a feasibility point in, the intersection of finitely many closed convex sets.  
In this section, we introduce backward and forward Bregman (pseudo-)circumcenter mappings and methods. We shall also investigate the convergence of the forward Bregman circumcenter method.   

Throughout this section, suppose that $f \in \Gamma_{0} (\mathcal{H}) $  with $\inte \dom f \neq \varnothing$ and that $f$ is G\^ateaux  differentiable  on $\inte \dom f$. Set $\I := \{1, \ldots, m\}$.

\subsection*{Bregman circumcenter mappings}
 
 Henceforth, for every $\mathcal{D} \subseteq \mathcal{H}$,
 \begin{align*}
 \text{denote by }   \mathcal{P}(\mathcal{D}) \text{ the set of all nonempty subsets of } \mathcal{D} \text{ containing finitely many elements}.
 \end{align*}
Suppose that  $K \in \mathcal{P} (\dom f)$. Set 
\begin{align*}
\overrightarrow{E}_{f}(K):= \{ q \in \inte \dom f~:~  (\forall x \in K) ~\D_{f} (x,q) \text{ is identical} \}.
\end{align*}
Suppose additionally $K \in \mathcal{P} (\inte \dom f)$. Denote by
\begin{align*}
\overleftarrow{E}_{f}(K):= \{ p \in \dom f ~:~ (\forall y \in K) ~ \D_{f} (p,y) \text{ is identical}\}.
\end{align*} 

 \begin{definition} {\rm \cite[Definition~3.1]{BOyW2021BregmanCircumBasicTheory}} \label{defn:CCS:Bregman:left}
 	Let $K \in \mathcal{P}( \inte \dom f)  $. 
 	\begin{enumerate}
 		\item \label{defn:CCS:Bregman:left:} Define the \emph{backward Bregman circumcenter operator $\overleftarrow{\CCO{}} $ w.r.t.\,$f$} as
 		\begin{align*}
 		\overleftarrow{\CCO{}}  : \mathcal{P}( \inte \dom f) \to 2^{\mathcal{H}} :  K \mapsto
 		\aff (K )\cap \overleftarrow{E}_{f}( K ).
 		\end{align*}
 		\item \label{defn:CCS:Bregman:left:ps}	Define the \emph{backward Bregman pseudo-circumcenter operator $\overleftarrow{\CCO{}}^{ps} $ w.r.t.\,$f$} as
 		\begin{align*}
 		\overleftarrow{\CCO{}}^{ps}  :  \mathcal{P}( \inte \dom f) \to 2^{\mathcal{H}} :  K \mapsto  \aff ( \nabla f (K ))\cap \overleftarrow{E}_{f}( K ).
 		\end{align*}
 	\end{enumerate}
 	In particular, for every $K \in  \mathcal{P}( \inte \dom f)  $, we call the element  in  $	\overleftarrow{\CCO{}} (K) $ and   $\overleftarrow{\CCO{}}^{ps}  (K)$ \emph{Backward Bregman  circumcenter  and  Backward Bregman  pseudo-circumcenter of $K$}, respectively.
 \end{definition}

 \begin{definition}  {\rm \cite[Definition~4.1]{BOyW2021BregmanCircumBasicTheory}} \label{defn:CCS:Bregman:forward}
 	Let $K \in \mathcal{P} (\dom f)$.
 	\begin{enumerate}
 		\item \label{defn:CCS:Bregman:forward:}	Define the \emph{forward Bregman circumcenter  operator w.r.t.\,$f$} as
 		\begin{align*}
 		\overrightarrow{\CCO{}}   :  	\mathcal{P} (\dom f) \to 2^{\mathcal{H}} : K \mapsto  \aff (  K ) \cap 	\overrightarrow{E}_{f}(K).
 		\end{align*}
 		\item \label{defn:CCS:Bregman:forward:ps} Define the \emph{forward Bregman pseudo-circumcenter  operator w.r.t.\,$f$} as
 		\begin{align*}
 		\overrightarrow{\CCO{}}^{ps}  :  \mathcal{P} (\dom f)  \to  2^{\mathcal{H}}  : K \mapsto  
 		\left(	\nabla f^{*} \left( \aff   K \right) \right) \cap \overrightarrow{E}_{f}  (K).
 		\end{align*}
 	\end{enumerate}
	In particular, for every $K \in  \mathcal{P} (\dom f)$, we call the element  in $\overrightarrow{\CCO{}} (K) $ and   $\overrightarrow{\CCO{}}^{ps} (K)$ \emph{forward Bregman circumcenter  and forward Bregman pseudo-circumcenter of $K$}, respectively. 
 \end{definition}

 Suppose  that $\D_{f}$ is symmetric, that is, $\dom f =\inte \dom f$ and $(\forall \{x,y\} \subseteq \dom f)$ $\D_{f} (x,y)=\D_{f} (y,x)$. Then the backward Bregman circumcenter and forward Bregman circumcenter are consistent. In particular, 
 if $f:= \frac{1}{2} \norm{\cdot}^{2}$, then $\nabla f =\Id$ and hence, the notions backward Bregman circumcenter,  backward Bregman pseudo-circumcenter, forward Bregman circumcenter and  forward Bregman pseudo-circumcenter  are all the same and reduce to the circumcenter defined in  \cite[Definition~3.4]{BOyW2018} under the Euclidean distance.
 
 From now on,  for every set-valued operator $A : \mathcal{H} \to 2^{\mathcal{H}}$, if $A(x)$ is a singleton for some $x \in \mathcal{H}$, we sometimes by slight abuse of notation allow   $A(x)$ to stand for  its unique element. The intended meaning should be clear from the context. 
   
 \begin{fact} {\rm \cite[Corollary~6.1]{BOyW2021BregmanCircumBasicTheory}} \label{fact:CircumOperator}
 	Suppose that $K:= \{q_{0}, q_{1}, \ldots, q_{m} \} \subseteq    \dom f$  is nonempty.
 		Then the following assertions hold.
 		\begin{enumerate}
 			\item \label{fact:CircumOperator:back} Suppose that $\mathcal{H} =\mathbb{R}^{n}$, that $f$ is Legendre such that $\dom f^{*}$ is open, that $\overleftarrow{E}_{f}(K)  \neq \varnothing$,  that $f$	allows forward Bregman projections,   that $\aff (K) \subseteq \inte \dom f$, and that $\nabla f (\aff (K)  ) $ is a closed  affine subspace.    Then $\left( \forall z \in \overleftarrow{E}_{f}(K) \right)$ $ \overrightarrow{\Pro}^{f}_{\aff(K )} (z) \in \overleftarrow{\CCO{}} (K)$.  
 			
 			\item \label{fact:CircumOperator:backps} Suppose that $\overleftarrow{E}_{f}(K) \neq \varnothing$ and that $K \subseteq \inte \dom f $ and  $\aff(\nabla f( K) ) \subseteq \dom f$.  Then $\left( \forall z \in \overleftarrow{E}_{f}(K) \right)$ $\overleftarrow{\CCO{}}^{ps} (K)=\Pro_{\aff(\nabla f( K ) )} (z)  $.   
 			
 			\item \label{fact:CircumOperator:backps:formu} Suppose that $K \subseteq \inte \dom f $ and $\aff ( \nabla f (K) ) \subseteq \dom f$, and that $\nabla f(q_{0}), \nabla f(q_{1}), \ldots, \nabla f(q_{m})$ are affinely independent.
 			Then $\overleftarrow{\CCO{}}^{ps} (K) $ uniquely exists and has an explicit formula. 
 			
 			\item \label{fact:CircumOperator:for} Suppose that $f$ is Legendre,  that $ \aff (K) \cap  \inte \dom f \neq \varnothing$, and that $\overrightarrow{E}_{f}(K )   \neq \varnothing$.  Then 
 			$\left( \forall z \in \overrightarrow{E}_{f}(K )  \right)$
 			$ \overleftarrow{\Pro}^{f}_{\aff(K)}(z) \in \overrightarrow{\CCO{}} (K)$.   
 			
 			\item \label{fact:CircumOperator:forps} Suppose that $f$ is Legendre, and that  $\aff( K  ) \subseteq  \inte \dom f^{*}$ and $\overrightarrow{E}_{f}(K )   \neq \varnothing$.  Then 
 			$\left( \forall z \in \overrightarrow{E}_{f}(K)  \right)$  $\overrightarrow{\CCO{}}^{ps} (K)= \nabla f^{*} \left( \Pro_{\aff( K  )} (\nabla f(z)) \right)$.  
 			
 			\item \label{fact:CircumOperator:forps:for}
 			Suppose that $f$ is Legendre, that $\aff ( K) \subseteq \inte \dom f^{*}$, and  that $q_{0}, q_{1}, \ldots, q_{m}$ are affinely independent. 
 			Then $\overrightarrow{\CCO{}}^{ps} (K) $ uniquely exists and has an explicit formula.  
 		\end{enumerate}
 \end{fact}

Note that we have particular examples in \cite{BOyW2021BregmanCircumBasicTheory} with functions $f \neq \frac{1}{2} \norm{ \cdot}^{2}$ and  sets $K$ such that the hypotheses of each item of \cref{fact:CircumOperator} above hold.

\begin{definition}  \label{defn:BregCirMap}
Let $t \in \mathbb{N} \smallsetminus \{0\}$. 	Suppose that $G_{1}, \ldots, G_{t}$ are operators from $\mathcal{H}$ to $\mathcal{H}$. 
Set 
	\begin{align*} 
	\mathcal{S}:= \{G_{1}, \ldots, G_{t} \} \quad \text{and} \quad  (\forall x \in \mathcal{H})  ~ \mathcal{S}(x):= \{G_{1}x, \ldots,G_{t}x \}. 
	\end{align*}
		The  \emph{forward Bregman  circumcenter mapping induced by $\mathcal{S}$} is
		\begin{empheq}[box=\mybluebox]{equation} \label{eq:FBCM}
		\overrightarrow{\CCO}_{\mathcal{S}} \colon \mathcal{H} \to 2^{\mathcal{H}} \colon y \mapsto 	\overrightarrow{\CCO}(\mathcal{S}(y)), 
		\end{empheq}
		that is, for every $y \in \mathcal{H}$, if the forward Bregman circumcenter of the set $\mathcal{S}(y)$ defined in \cref{defn:CCS:Bregman:forward}\cref{defn:CCS:Bregman:forward:}	  does not exist, then  $\overrightarrow{\CCO}_{\mathcal{S}} y= \varnothing $. Otherwise, $\overrightarrow{\CCO}_{\mathcal{S}} y$ is the  set of points $v$ satisfying the two conditions below:
		\begin{enumerate}
			\item $v \in \aff(\mathcal{S}(x)) \cap \inte \dom f$, and
			\item $\D_{f} (G_{1}x,v)   =\cdots =\D_{f} (G_{t}x,v) $.
		\end{enumerate}
Let $\mathcal{C}$ be a subset of $\mathcal{H}$.
If $(\forall y \in  \mathcal{C})$ $\overrightarrow{\CCO}_{\mathcal{S}}y$   contains at most one element, we say that $\overrightarrow{\CCO}_{\mathcal{S}}$ is \emph{at most single-valued on $\mathcal{C}$}. Naturally,   if $\mathcal{C} =\mathcal{H}$, then we omit the phrase \enquote{on $\mathcal{C}$}.

Analogously, we define the \emph{forward Bregman pseudo-circumcenter mapping $\overrightarrow{\CCO}_{\mathcal{S}}^{ps} $ induced by $\mathcal{S}$}, \emph{backward Bregman  circumcenter mapping $\overleftarrow{\CCO}_{\mathcal{S}} $ induced by $\mathcal{S}$}, and \emph{backward Bregman pseudo-circumcenter mapping $\overleftarrow{\CCO}_{\mathcal{S}}^{ps} $ induced by $\mathcal{S}$}  with replacing the forward Bregman circumcenter operator $\overrightarrow{\CCO}$ in \cref{eq:FBCM} by the corresponding operator, $\overrightarrow{\CCO}^{ps}$, $\overleftarrow{\CCO}$, and $\overleftarrow{\CCO}^{ps}$, respectively. 
\end{definition}

We can also directly deduce the following result by \cite[Theorem~3.3(ii)]{BOyW2019Isometry}.
\begin{corollary}
	Suppose that  $f :=\frac{1}{2} \norm{\cdot}^{2}$.  Let $(\forall i \in \I)$ $T_{i} : \mathcal{H} \rightarrow \mathcal{H}$ be a linear  isometry\footnotemark.
	\footnotetext{A mapping $T: \mathcal{H} \to \mathcal{H}$ is said to be isometric or an isometry if $(\forall x \in \mathcal{H})$ $(\forall y \in \mathcal{H})$ $\norm{Tx- Ty} =\norm{x-y}$.}
	Set $\mathcal{S} :=\{ T_{1}, \ldots,   T_{m} \} $ and $(\forall x \in \mathcal{H})$  $\mathcal{S}(x) = \{ T_{1}x, \ldots,   T_{m}x  \}$. Then	$\overrightarrow{\CCO}_{\mathcal{S}}  (x)=\overrightarrow{\CCO}_{\mathcal{S}}^{ps} (x) = \overleftarrow{\CCO}_{\mathcal{S}} (x)= \overleftarrow{\CCO}_{\mathcal{S}}^{ps} (x)=\Pro_{\aff(\mathcal{S}(x))} (0)$.
\end{corollary}

\begin{proof}
		Notice that $\nabla f =\Id$ and $\dom f =\mathcal{H}$ in this case. 
	Then
		$0 \in \cap^{m}_{j=1} \Fix T_{j} \subseteq \overrightarrow{E}_{f}(K) =\overleftarrow{E}_{f}(K) $. Therefore,	the required result follows easily from    \cref{fact:CircumOperator}\cref{fact:CircumOperator:backps}   or \cref{fact:CircumOperator}\cref{fact:CircumOperator:forps}.
\end{proof}	

Henceforth,  suppose 
$T_{1}, \ldots, T_{m}$ are operators from $\mathcal{H}$ to $\mathcal{H}$ with $ \inte \dom f  \cap \left( \cap_{i \in \I} \Fix T_{i}    \right) \neq \varnothing $. Unless stated otherwise, we always set  
\begin{empheq}[box=\mybluebox]{equation} \label{eq:SSx}
\mathcal{S}:= \{\Id, T_{1}, \ldots, T_{m} \}  \quad  \text{and}\quad    (\forall x \in \mathcal{H})  ~ \mathcal{S}(x):= \{x, T_{1}x, \ldots, T_{m}x \}. 
\end{empheq}

\begin{lemma} \label{lemma:Fix}
	Suppose that $f$ is essentially strictly convex. Then the following statements hold.
	\begin{enumerate}
		\item \label{lemma:Fix:forward} $ \Fix \overrightarrow{\CCO}_{\mathcal{S}} = \inte \dom f \cap  \left(  \cap_{i \in \I} \Fix T_{i} \right)$ and $ \Fix \overrightarrow{\CCO}_{\mathcal{S}}^{ps} = \inte \dom f \cap   \left(  \cap_{i \in \I} \Fix T_{i} \right) \cap  \Fix   \nabla f^{*} $.
		\item \label{lemma:Fix:backward} $\Fix \overleftarrow{\CCO}_{\mathcal{S}} =\inte \dom f \cap \left( \cap_{i \in \I} \Fix T_{i} \right)$ and  $\Fix \overleftarrow{\CCO}_{\mathcal{S}}^{ps} 
		=\left( \cap_{i \in \I} \Fix T_{i} \right) \cap \Fix   \nabla f   $.
	\end{enumerate} 
\end{lemma}

\begin{proof}
	\cref{lemma:Fix:forward}: Let $z \in \mathcal{H}$.	
	\begin{align*}
	z \in \Fix \overrightarrow{\CCO}_{\mathcal{S}}  \Leftrightarrow & z \in  \overrightarrow{\CCO}_{\mathcal{S}}z  \\
	\Leftrightarrow & z \in \aff(\mathcal{S}(z)) \cap \inte \dom f  \text{ and  } \D_{f} (z,z) =\D_{f} ( T_{1}z,z)   =\cdots =\D_{f} (T_{m}z,z)\\ 	
	\Leftrightarrow & z \in  \inte \dom f\text{ and  }(\forall i \in \I) ~z \in \Fix T_{i} \quad (\text{by \cref{fact:PropertD}})\\ 
	\Leftrightarrow & z \in \inte \dom f \cap  \left(  \cap_{i \in \I} \Fix T_{i} \right),
	\end{align*}	
	which guarantees that $ \Fix \overrightarrow{\CCO}_{\mathcal{S}} = \inte \dom f \cap  \left(  \cap_{i \in \I} \Fix T_{i} \right)$. 
	
	Similarly, 
	\begin{align*}
	z \in \Fix  \overrightarrow{\CCO}_{\mathcal{S}}^{ps}	\Leftrightarrow &  z = \overrightarrow{\CCO}_{\mathcal{S}}^{ps}z    \\
	\Leftrightarrow & z \in \nabla f^{*} \left( \aff \left(\mathcal{S} (z) \right) \right)\cap \inte \dom f   \text{ and  } \D_{f} (z,z) =\D_{f} ( T_{1}z,z)   =\cdots =\D_{f} (T_{m}z,z)\\ 	
	\Leftrightarrow & z \in \Fix \left( \nabla f^{*} \left( \aff  ( \mathcal{S}) \right) \right)  \cap  \inte \dom f\text{ and  }(\forall i \in \I) ~z \in \Fix T_{i} \quad (\text{by \cref{fact:PropertD}})\\ 
	\Leftrightarrow & z \in   \inte \dom f \cap  \left(  \cap_{i \in \I} \Fix T_{i} \right) \cap  \Fix \left( \nabla f^{*} \left( \aff  ( \mathcal{S}) \right) \right),
	\end{align*}
	which deduces that $\Fix \overrightarrow{\CCO}_{\mathcal{S}}^{ps} = \inte \dom f \cap   \left(  \cap_{i \in \I} \Fix T_{i} \right) \cap  \Fix \left( \nabla f^{*} \left( \aff  ( \mathcal{S}) \right) \right) $.
	
	In addition, clearly,  $\Id \in \mathcal{S}$ entails that $ \Fix  \nabla f^{*}  \subseteq   \Fix \left( \nabla f^{*} \left( \aff  ( \mathcal{S}) \right) \right)$. Hence, $ \left(  \cap_{i \in \I} \Fix T_{i} \right) \cap \Fix  \nabla f^{*}  \subseteq  \left(  \cap_{i \in \I} \Fix T_{i} \right) \cap  \Fix \left( \nabla f^{*} \left( \aff  ( \mathcal{S}) \right) \right)$. On the other hand, $\left(\forall y \in   \left(  \cap_{i \in \I} \Fix T_{i} \right) \cap  \Fix \left( \nabla f^{*} \left( \aff  ( \mathcal{S}) \right) \right) \right)$ $y \in \nabla f^{*} \left( \aff  ( \mathcal{S}(y)) \right) =  \nabla f^{*}(y)$, which necessitates that $\left(  \cap_{i \in \I} \Fix T_{i} \right) \cap  \Fix \left( \nabla f^{*} \left( \aff  ( \mathcal{S}) \right) \right) \subseteq  \left(  \cap_{i \in \I} \Fix T_{i} \right) \cap \Fix  \nabla f^{*} $. Therefore, $\left(  \cap_{i \in \I} \Fix T_{i} \right) \cap  \Fix \left( \nabla f^{*} \left( \aff  ( \mathcal{S}) \right) \right)=  \left(  \cap_{i \in \I} \Fix T_{i} \right) \cap \Fix  \nabla f^{*} $. 
	
	Altogether, $\Fix \overrightarrow{\CCO}_{\mathcal{S}}^{ps} = \inte \dom f \cap   \left(  \cap_{i \in \I} \Fix T_{i} \right) \cap \Fix  \nabla f^{*} $.
	
	\cref{lemma:Fix:backward}: The proof is similar to the proof of \cref{lemma:Fix:forward} above. 
\end{proof}

 \subsection*{Bregman isometry and Browder's Demiclosedness Principle}
 In this subsection, we generalize the traditional isometry and Browder's Demiclosedness Principle from the Euclidean distance to general Bregman distances and 
investigate the Bregman isometry and Browder's Demiclosedness Principle, which play critical roles in the main result in this section. The Bregman isometry and Browder's Demiclosedness Principle are interesting in their own right.
 \begin{definition} \label{defn:BregIsometry}
 	Let $\mathcal{C} $ be a nonempty subset of $\mathcal{H}$ and let $T: \mathcal{C} \to \mathcal{H} $. We say  $T$ is \emph{Bregman isometric} (or a \emph{Bregman isometry}) w.r.t. $f$, if  $\left( \forall (x,y) \in \dom f \times \inte \dom f  \right)$ $\left( Tx,Ty \right) \in \dom f \times \inte \dom f $, and
 	\begin{align} \label{eq:defn:BregIsometry}
 	\left(\forall \{x,y\}  \subseteq \dom f \times \inte \dom f \right) ~ \D_{f} \left( Tx, T y   \right) = \D_{f}(x,y).
 	\end{align} 
 \end{definition}	
 In view of the definition, it is trivial that given an arbitrary function $g$ satisfying the statements of \cref{defn:BregmanDistance},  the identity operator $\Id$ is Bregman isometric w.r.t. $g$.
 
 Notice that if $f= \frac{1}{2}\norm{\cdot}^{2}$, then the Bregman isometry deduces the traditional isometry $($see, e.g., \cite[Lemma~2.23]{BOyW2019Isometry} for examples of isometries under the Euclidean distance$)$.

 	 The following definition is a generalization of the well-known Browder's Demiclosedness Principle \cite[Theorem~3(a)]{Browder1968}.  In view of \cite[Corollary~4.25]{BC2017}, when $f =\frac{1}{2} \norm{\cdot}^{2}$, the Bregman Browder's Demiclosedness Principle  holds for all nonexpansive operators. 
 
 \begin{definition} \label{defn:Browder}
 	Let $\mathcal{C} $ be a nonempty subset of $\mathcal{H}$, let $T: \mathcal{C} \to \mathcal{H} $ and let $x \in \inte \dom f$. We say the \emph{Bregman Browder's demiclosedness principle associated with $f$ holds at $x \in C \cap \inte \dom f$ for $T$} if for every sequence $\left( x_{k} \right)_{k \in \mathbb{N}} $ in $ \inte \dom f$, $ \left( T x_{k} \right)_{k \in \mathbb{N}} $ in $ \inte \dom f$ and
 	\begin{align} \label{eq:defn:Browder}
 	\begin{rcases}
 	x_k \weakly  x\\
 	\D_{f}(x_k, Tx_k) \to 0
 	\end{rcases}
 	\Rightarrow x \in \Fix T.	
 	\end{align} 
 	In addition, we say   \emph{$T$ is $f$-demiclosed}, if the Bregman Browder's demiclosedness principle associated with $f$ holds for every $x \in C \cap \inte \dom f$ for $T$. 
 \end{definition}
 
 Because $\Fix \Id = \mathcal{H}$, given an arbitrary function $g$  satisfying the statements of \cref{defn:BregmanDistance}, $\Id$  is $g$-demiclosed.

 The following examples are used to illustrate the two new concepts above. 
 \begin{example}  \label{prop:Texamples}
 	Suppose that $\mathcal{H} =\mathbb{R}^{n}$. 	Denote by  $(\forall x \in \mathbb{R}^{n} )$ $x= (x_{i})^{n}_{i=1} $. Define $f :x \mapsto   -\sum^{n}_{i=1}   \ln (x_{i})  $, with $\dom f =\left]0, +\infty\right[^{n} $, which is the Burg entropy. 
 	
 	Then the following statements hold. 
 	
 	\begin{enumerate}
 			\item  \label{prop:Texamples:f} $f$ is Legendre  such that $\dom f $ and $\dom f^{*}$ are open. 
 			
 		\item \label{prop:Texamples:ci}  Set $\J := \{ 1, \ldots,n  \}$.
 		Let $(c_{i})_{i \in \J}  \in \mathbb{R}^{n}$ with $(\forall i \in \J)$ $c_{i} \in \mathbb{R}_{++}$. Define $T : \mathbb{R}^{n} \to \mathbb{R}^{n} : (x_{i})_{i \in \J}  \mapsto (c_{i} x_{i})_{i \in \J} $. Then:
 		\begin{enumerate}			  
 			\item \label{prop:Texamples:isometry} $T$ is Bregman isometric w.r.t. $f$.
 			\item  \label{prop:Texamples:browder} Suppose that there exists $  j \in \J$  such that $c_{j} \neq 1 $. Then $T$ is $f$-demiclosed.
 		\end{enumerate}
 	\item Suppose that   $\mathcal{H} =\mathbb{R} $. Let $c \in \mathbb{R}_{++}$ and $\mu \in \mathbb{R}_{++} \smallsetminus \{1\}$. Define $T : \mathbb{R}  \to \mathbb{R} : t \mapsto c t^{\mu} $. Then: 
 	\begin{enumerate}
 		\item \label{prop:Texamples:cap}  $\inte \dom f \cap   \Fix T = \{ c^{- \frac{1}{\mu-1}}\}$, which is nonempty closed and convex.
 		\item \label{prop:Texamples:BBDP}  $T$ is $f$-demiclosed.  		
 	\end{enumerate}
 	\end{enumerate}
 	
 \end{example}
 
 \begin{proof}
 	\cref{prop:Texamples:f}: This is immediately from \cite[Examples~2.1]{BC2003}.

 	\cref{prop:Texamples:isometry}: This is clear from  \cref{defn:BregmanDistance} and the definitions of $f$ and $T$.
 	
 	 \cref{prop:Texamples:browder}: Let $\left( x^{(k)} \right)_{k \in \mathbb{N}} $ be in $\dom f$.
 	 Because  $(\forall x \in  \dom f)$ $Tx \in    \dom f$, we have $\left( T x^{(k)} \right)_{k \in \mathbb{N}} $ is a sequence in $\dom f$.

Suppose that $x^{(k)} \to \bar{x}$ with $\bar{x} := \left( \bar{x}_{i} \right)_{i \in \J} \in \left]0, +\infty\right[^{n}$. Denote by $(\forall k \in \mathbb{N})$ $x^{(k)} := \left( x^{(k)}_{i} \right)_{i \in \J} $. Now,
\begin{align*}
	\D_{f}(x^{(k)}, Tx^{(k)}) & = f (x^{(k)}) - f (T x^{(k)}) - \innp{\nabla f ( T x^{(k)}) , x^{(k)} -T x^{x}} \\
	& =  -\sum^{n}_{i=1}   \ln (x^{(k)}_{i})  + \sum^{n}_{i=1}   \ln (c_{i }x^{(k)}_{i}) + \sum^{n}_{i=1} \frac{x^{(k)}_{i}  -c_{i }x^{(k)}_{i} }{c_{i }x^{(k)}_{i} } \\
	& =  \sum^{n}_{i=1}   \ln \frac{ c_{i }x^{(k)}_{i} }{ x^{(k)}_{i}}   + \sum^{n}_{i=1} \frac{x^{(k)}_{i}  -c_{i }x^{(k)}_{i} }{c_{i }x^{(k)}_{i} },
\end{align*}
which implies that 
\begin{align*}
\D_{f}(x^{(k)}, Tx^{(k)})  \to \sum^{n}_{i=1}  \left(   \ln  ( c_{i } )    +   \frac{ 1 }{c_{i } } -1  \right).
\end{align*}
On the other hand,  consider the function $g : \mathbb{R}_{++} \to \mathbb{R} : t \mapsto \ln (t) + \frac{1}{t} -1 $.  By some easy calculus,  $(\forall  t  \in \mathbb{R}_{++}  \smallsetminus \{1\})$ $ g(t)  > g(1) =0$. Because there exists $  j \in \J$  such that $c_{j} \neq 1 $, we know that $\sum^{n}_{i=1}  \left(   \ln  ( c_{i } )    +   \frac{ 1 }{c_{i } } -1  \right)  > 0$. 
Thus it never holds that $\D_{f}(x^{(k)}, Tx^{(k)}) \to 0$, and so, vacuously, \cref{eq:defn:Browder} holds.

\cref{prop:Texamples:cap}: This is trivial. 

 \cref{prop:Texamples:BBDP}: Let $a \in \left]0, +\infty\right[\,$. Let $( a_{k})_{k \in \mathbb{N}} $ be in $ \inte \dom f$ such that $a_{k} \to a$. 
 Then clearly $ ( Ta_{k})_{k \in \mathbb{N}} $ is a sequence in $  \left]0, +\infty\right[\,$. Suppose that $ \D_{f}(a_k, Ta_k) \to 0$. Note that 
 \begin{align*}
 \D_{f}(a_k, Ta_k)  = - \ln(a_{k}) + \ln (c a_{k}^{\mu}) +\frac{a_{k} - c a_{k}^{\mu}}{ c a_{k}^{\mu}} \to   \ln(ca^{\mu}) -\ln(a) + \frac{a^{1-\mu}}{c} -1. 
 \end{align*}
 Denote by  $h: \mathbb{R}_{++} \to \mathbb{R} : t \mapsto\ln(ct^{\mu}) -\ln(t) + \frac{t^{1-\mu}}{c} -1$.  Then $(\forall t \in \mathbb{R}_{++} )$ $h'(t) =(\mu-1) (\frac{1}{t} - \frac{1}{ct^{\mu}})$.
 Then by considering the two cases $\mu >1$ and $\mu <1$ separately, we easily get that 
 $\ln(ca^{\mu}) -\ln(a) + \frac{a^{1-\mu}}{c} -1=0$ implies that $a=ca^{\mu}$, that is, $a =c^{- \frac{1}{\mu-1}}  \in \Fix T$. Altogether, by \cref{defn:Browder},  $T$ is $f$-demiclosed.
 \end{proof}
Consider our examples $f$ and $T$ in \cref{prop:Texamples}\cref{prop:Texamples:ci}.  Notice that if there exists $  j \in \J$  such that $c_{j} \neq 1 $,  then  $\inte \dom f \cap  \left(  \cap_{i \in \I} \Fix T_{i} \right) =\varnothing$. 

The following example of $f$ and $T$ satisfies all requirements in   \cref{theorem:xkConvergence} below.
\begin{example} \label{prop:FDE}
	Suppose that $\mathcal{H} =\mathbb{R}^{n}$. Denote by  $(\forall x \in \mathbb{R}^{n} )$ $x= (x_{i})^{n}_{i=1} $. 	 Define $f :x \mapsto   \sum^{n}_{i=1} x_{i} \ln  (x_{i}) + (1-x_{i} ) \ln(1-x_{i} )$, with $\dom f =\left[0, 1\right]^{n} $, which is the Fermi-Dirac entropy. 
	\footnotemark
	\footnotetext{Here and elsewhere, we use the convention that $0 \ln (0) =0$.}
	 Set $\J := \{ 1, \ldots,n  \}$. 
	Then the following statements hold.
	\begin{enumerate}
		\item  \label{prop:FDE:f} $f$ is Legendre with $\dom f^{*}$ open, and allows forward Bregman projections.
		
		\item \label{prop:FDE:BBDP:Id}   $\Id$  is $f$-demiclosed.  	 
		\item Define $T: \mathbb{R}^{n} \to \mathbb{R}^{n} : (x_{i})^{n}_{i=1}  \mapsto (1- x_{i})^{n}_{i=1}  $. Then:
	\begin{enumerate}
	\item  \label{prop:FDE:isometry}  $T$ is Bregman isometric w.r.t. $f$.
	\item   \label{prop:FDE:BBDP} $T$ is $f$-demiclosed.  		
	\item  \label{prop:FDE:FixT} $ \Fix T \cap \inte \dom f  =\{ (\frac{1}{2})_{i \in \J} \} $ is nonempty, closed and convex. 
	\end{enumerate}
	\item \label{prop:FDE:Ti} Let $\Lambda $ be a subset of $J$. Define $T: \mathbb{R}^{n} \to \mathbb{R}^{n}$ by $\left( \forall x  \in \mathbb{R}^{n} \right)$ $T x := (y_{i})^{n}_{i=1}$ where  $(\forall i \in \J \smallsetminus \Lambda )$ $ y_{i} =x_{i}$  and $(\forall i \in   \Lambda )$ $y_{i} = (1-x_{i})$. 
		\begin{enumerate}
\item  \label{prop:FDE:TiFixT} $\Fix T = \left\{ x \in \mathbb{R}^{n} ~:~  (\forall i \in \Lambda)   x_{i} = \frac{1}{2}  \right\}$.
\item  \label{prop:FDE:TiBI} $T$ is Bregman isometric w.r.t. $f$.

\item  \label{prop:FDE:TiBBDP} $T$ is $f$-demiclosed.  		
	\end{enumerate}
			\end{enumerate}
\end{example}

\begin{proof}
	\cref{prop:FDE:f}:  This follows immediately from \cite[Examples~2.1 and 2.7]{BC2003}.
	
\cref{prop:FDE:BBDP:Id}: This is trivial. 
	
	\cref{prop:FDE:isometry}: This follows easily from  \cref{defn:BregmanDistance} and the definitions of $f$ and $T$.
	
	\cref{prop:FDE:BBDP}:  Let $\bar{x} \in \left] 0, 1 \right[^{n} $ and let $( x^{(k)} )_{k \in \mathbb{N}} $ be a sequence in $\left] 0, 1 \right[^{n} $ such that $x^{(k)} \to \bar{x}$. Denote by  $\bar{x} := (\bar{x} _{i})^{n}_{i=1} $ and  $(\forall k \in \mathbb{N})$ $x^{(k)} = (x^{(k)}_{i})_{i \in \J} $.  In view of the definition of $T$, 
	 $( Tx^{(k)} )_{k \in \mathbb{N}} $ is also a sequence in  $\left] 0, 1 \right[^{n} $.
	 Suppose that $\D_{f}(x_k, Tx_k) \to 0$. Notice that 
	 \begin{align*}
	 \D_{f}(x^{(k)}, Tx^{(k)} ) & =  \sum^{n}_{i=1} \left(x^{(k)}_{i} \ln  (x^{(k)}_{i}) + (1-x^{(k)}_{i} ) \ln(1-x^{(k)}_{i} ) \right) -   \sum^{n}_{i=1} \left(  (1-x^{(k)}_{i} ) \ln(1-x^{(k)}_{i} ) +x^{(k)}_{i} \ln  (x^{(k)}_{i})  \right) \\
	 & \quad ~ - \sum^{n}_{i=1} \left(x^{(k)}_{i} - ( 1-x^{(k)}_{i})  \right)  \ln \left(  \frac{  1-x^{(k)}_{i} }{ x^{(k)}_{i} }\right) \\
	 &= - \sum^{n}_{i=1} \left(x^{(k)}_{i} - ( 1-x^{(k)}_{i})  \right)  \ln \left(  \frac{  1-x^{(k)}_{i} }{ x^{(k)}_{i} }\right)  \to    - \sum^{n}_{i=1} \left(\bar{x}_{i} - ( 1-\bar{x}_{i})  \right)  \ln \left(  \frac{  1-\bar{x}_{i} }{ \bar{x}_{i} }\right).
	 \end{align*}  
	 Consider the function $g: \left] 0, 1 \right[ \to \mathbb{R}: t \mapsto  - \left( t - ( 1-t  )  \right)  \ln \left(  \frac{  1- t }{ t }\right)$. Since $(\forall t \in  \left] 0, 1 \right[\,)$ $g''(t) = \frac{1}{(t-1)^{2}t^{2}} >0$ and $g'(\frac{1}{2}) =0$,  thus $\left(\forall t \in  \left] 0, 1 \right[ \smallsetminus \{ \frac{1}{2} \}  \right)$ $g(t) > 0$. Hence,   $  - \sum^{n}_{i=1} \left(\bar{x}_{i} - ( 1-\bar{x}_{i})  \right)  \ln \left(  \frac{  1-\bar{x}_{i} }{ \bar{x}_{i} }\right) =0$ implies that $\bar{x} =(\frac{1}{2})_{i \in \J} \in \Fix T$. Therefore, the assertion is true. 
	 
	 	\cref{prop:FDE:FixT}: The required results are trivial. 
	 	
	 \cref{prop:FDE:Ti}: According to the definitions of $T$ and $f$, it is easy to see that 
	 \begin{align} \label{eq:prop:FDE}
	 (\forall x \in \dom f) (\forall y \in \dom f) \quad f(x) -f(Tx) =0   \text{ and }   \Innp{\nabla f (Ty), Tx -Ty} =\Innp{\nabla f (y),  x - y}. 
	 \end{align}
	 
	 \cref{prop:FDE:TiFixT}: This is clear. 
	  
\cref{prop:FDE:TiBI}: This follows immediately from \cref{defn:BregIsometry} and \cref{defn:BregmanDistance}.
	 
\cref{prop:FDE:TiBBDP}: Applying a proof similar to that of	\cref{prop:FDE:BBDP} and invoking  \cref{prop:FDE:TiFixT}, we obtain the required result.
\end{proof}

\subsection*{Convergence of forward Bregman circumcenter methods} 

Motivated by the \cref{lemma:Fix},    to find a point in the intersection $      \cap_{i \in \I} \Fix T_{i}  $, we prefer the backward and forward Bregman circumcenter mappings induced by $\mathcal{S}$ rather than the backward and forward Bregman pseudo-circumcenter mappings induced by $\mathcal{S}$ with smaller fixed point sets.  
Moreover, notice that comparing with the hypothesis of  \cref{fact:CircumOperator}\cref{fact:CircumOperator:back} on the existence of the backward Bregman circumcenter, we don't have a strong requirement for $f$ in the   \cref{fact:CircumOperator}\cref{fact:CircumOperator:for} on the existence of the forward Bregman circumcenter.   
Therefore, we consider only  the convergence of sequences of iterations generated by forward Bregman circumcenter mappings in this work.

The \emph{forward Bregman circumcenter method} generates the sequence of iterations of the forward Bregman circumcenter  mapping.

According to \cite[Theorem~4.3(ii)]{BOyW2018Proper} and \cite[Theorem~3.3(ii)]{BOyW2019Isometry}, the circumcenter mappings induced by finite sets of isometries (see \cite[Definition~2.27]{BOyW2019Isometry} for the exact definition), are special operators $G$ satisfying conditions in \cref{theorem:conver:G} below.
\begin{theorem}\label{theorem:conver:G}
	Suppose that $f  $ is  Legendre.
	Let $\mathcal{C}$ be a nonempty subset of $\inte \dom f$. Suppose that $G: \mathcal{C} \to \mathcal{C} $ satisfies that $\Fix G \cap \inte \dom f \neq \varnothing$ and $\Fix G$ is a closed convex subset of $\mathcal{H}$, and that $(\forall x \in \mathcal{C} \cap \inte \dom f )$ $( \forall z \in \Fix G \cap \inte \dom f )$ $ G (x)  =\overleftarrow{\Pro}^{f}_{\aff(\mathcal{S}(x))} z $. Let $x \in \mathcal{C} \cap \inte \dom f$.  Then the following statements hold.
	\begin{enumerate}
		\item \label{theorem:conver:G:sequence} $(G^{k}x )_{k \in \mathbb{N}}$ is a well-defined sequence in $\mathcal{C} \cap \inte \dom f$.
		\item \label{theorem:conver:G:D}  $(G^{k}x )_{k \in \mathbb{N}}$  is forward Bregman  monotone with respect to $   \Fix G $. Consequently, $	\left( \forall z \in  \Fix G \cap \inte \dom f \right)$ $ \left(  \D_{f} \left( G^{k}x ,z \right)  \right)_{k \in \mathbb{N}} $ converges.  
		\item \label{theorem:conver:G:limD} $\lim_{k \to \infty}$ $ \D_{f} \left(  G^{k}x, G^{k+1}x   \right) =0$.
		\item  \label{theorem:conver:G:conv} Suppose that $G$ is $f$-demiclosed  and that all weak sequential cluster points of $(G^{k}x)_{k \in \mathbb{N}}$ lie in $\inte \dom f$. Then $(G^{k}x)_{k \in \mathbb{N}}$ weakly converges to a point in $\inte \dom f \cap \Fix G$.  
	\end{enumerate}
\end{theorem}

\begin{proof}
	Invoking the definition of the operator $G: \mathcal{C} \to \mathcal{C} $ and \cref{fact:Pleft}, we observe that
	\begin{align}\label{eq:theorem:conver:G:def}
	(\forall y \in  \mathcal{C} \cap \inte \dom f )  ~( \forall z \in \Fix G \cap \inte \dom f ) \quad G (y)  =\overleftarrow{\Pro}^{f}_{\aff(\mathcal{S}(y))} z \in  \mathcal{C} \cap \inte \dom f.
	\end{align}
	Hence, use $\{ y, G(y)  \} \subseteq \aff(\mathcal{S}(y))$ and 
	apply \cite[Theorem~2.1(i)]{BOyW2021BregmanCircumBasicTheory}  with $U=\aff\left(\mathcal{S}(x) \right)$ to yield that  
	\begin{align}\label{eq:theorem:conver:G}
	(\forall y \in \mathcal{C} \cap \inte \dom f )~	(\forall z \in \Fix G \cap \inte \dom f)  \quad  \D_{f} (y ,z) = \D_{f} (y,G(y))  +\D_{f} (G(y),z). 
	\end{align}
	
	\cref{theorem:conver:G:sequence}: This is clear from \cref{eq:theorem:conver:G:def} by induction.
	
	\cref{theorem:conver:G:D}:  Taking \cref{eq:theorem:conver:G} and \cref{defn:BregmanDistance} into account, we deduce that 
	\begin{align}  
	(\forall y \in \mathcal{C} \cap \inte \dom f )~	\left(\forall z \in \inte \dom f \cap  \Fix G \right)  \quad  \D_{f} (y,z)  \geq \D_{f} \left(  Gy ,z \right) .
	\end{align}
	Employing  \cref{theorem:conver:G:sequence}  above and \cref{defn:BregBackMonotone}, we know that $(G^{k}x )_{k \in \mathbb{N}}$  is forward Bregman  monotone with respect to $   \Fix G $, which, combining with \cref{theorem:ConvBregBackwMonot}\cref{theorem:ConvBregBackwMonot:D}, implies the convergence assertion.

	\cref{theorem:conver:G:limD}:  Let $z \in \Fix G \cap \inte \dom f$.	For every $k \in \mathbb{N}$, substitute $y$  in \cref{eq:theorem:conver:G} by $G^{k}x$ to deduce that $ \D_{f} \left(G^{k}x,z \right) =  \D_{f} \left( G^{k}x, G^{k+1}x  \right) +\D_{f} \left(  G^{k+1}x ,z \right)$, which implies that
	\begin{align*}
	\sum_{k \in \mathbb{N}} \D_{f} \left(  G^{k}x, G^{k+1}x  \right) =  \D_{f} \left( x,z \right) - \lim_{t \to \infty} \D_{f} \left( G^{t}x ,z \right) < \infty,
	\end{align*}
	where the existence of the limit $\lim_{t \to \infty} \D_{f} \left( G^{t}x ,z \right) $  is from  \cref{theorem:conver:G:D}. Therefore,   \cref{theorem:conver:G:limD} holds.

	\cref{theorem:conver:G:conv}:  Let $\bar{z} $ be a weak sequential cluster point of $(G^{k}x )_{k \in \mathbb{N}}$.  Then there exists a subsequence 
	$(G^{k_{j}}x  )_{j \in \mathbb{N}}$ of $(G^{k}x  )_{k \in \mathbb{N}}$ such that $G^{k_{j}}x  \weakly \bar{z}$.
	Notice that, in view of \cref{theorem:conver:G:limD} above, $\D_{f} \left(  G^{k}x, G^{k+1}x \right)    \to 0$.
	Due to the assumption, $\bar{z} \in \inte \dom f$. Hence, utilizing the assumption that  $G$ is $f$-demiclosed, we derive that
	\begin{align*}
	\begin{rcases}
	G^{k_{j}}x    \weakly  \bar{z}\\
	\D_{f} \left( G^{k_{j}}x ,G\left( G^{k_{j} }x  \right) \right) \to 0
	\end{rcases}
	\Rightarrow \bar{z}  \in \Fix G ,
	\end{align*} 
	which implies that all weak sequential cluster points of  $(G^{k}x  )_{k \in \mathbb{N}}$  lie in $ \Fix G  \cap \inte \dom f  $, since $\bar{z}$ is an arbitrary cluster point of $(G^{k}x  )_{k \in \mathbb{N}}$. 
	Therefore, via  \cref{theorem:ConvBregBackwMonot}\cref{theorem:ConvBregBackwMonot:conv},  $(G^{k}x  )_{k \in \mathbb{N}}$ weakly  converges to some point in $\inte \dom f \cap \Fix G$.
\end{proof}

 It is clear that the  characterization \cref{eq:theorem:conver:G} of the backward Bregman projection $\overleftarrow{\Pro}^{f}_{\aff(\mathcal{S}(y))} z$ plays an essential role in the proof of \cref{theorem:conver:G} which is critical to prove our main result \cref{theorem:Convergence} below.  This combined with \cite[Theorem~2.1(ii)]{BOyW2021BregmanCircumBasicTheory} suggests  that the idea of the proof of \cref{theorem:conver:G} does not work if we redefine $G: \mathcal{C} \to \mathcal{C} $ in \cref{theorem:conver:G}  by $(\forall x \in \mathcal{C} \cap \inte \dom f )$ $( \forall z \in \Fix G \cap \inte \dom f )$ $ G (x)  =\overrightarrow{\Pro}^{f}_{\aff(\mathcal{S}(x))} z $. On the other hand, in view of \cref{fact:CircumOperator}\cref{fact:CircumOperator:back}$\&$\cref{fact:CircumOperator:for}, under some conditions, certain backward Bregman projections (resp.\,forward Bregman projections) are forward Bregman circumcenters (resp.\,backward Bregman circumcenters). Therefore, it is easier to work on the forward Bregman circumcenter method than the backward Bregman circumcenter method.

The following \cref{theorem:Convergence}\cref{theorem:Convergence:Browder:CCS} and  \cref{theorem:Convergence}\cref{theorem:Convergence:Browder:Cluster} reduce to   \cite[Theorem~3.17]{BOyW2018Proper} and \cite[Theorem~4.7]{BOyW2019Isometry}, respectively,  when $f=\frac{1}{2} \norm{\cdot}^{2}$.  Notice that the circumcenter mappings studied in \cite{BOyW2019Isometry}, \cite{BOyW2019LinearConvergence}, \cite{BOyW2020BAM}, and \cite{Ouyang2020Finite} under the Euclidean distance are all single-valued operators, and that all of the examples of backward and forward Bregman (pseudo)-circumcenters presented in \cite{BOyW2021BregmanCircumBasicTheory} are singletons. So, our assumption \enquote{$\overrightarrow{\CCO}_{\mathcal{S}}$ is at most single-valued on $\inte \dom f$} in the following \cref{theorem:Convergence} is not too restrictive. In addition, it is clear that if $f= \frac{1}{2}\norm{\cdot}^{2}$, then   the following condition \cref{eq:triangleineq} is a direct result from the triangle inequality.  

\begin{theorem} \label{theorem:Convergence}
	Suppose that $f$ is Legendre,   that 	$\inte \dom f \cap  \left(  \cap_{i \in \I} \Fix T_{i} \right)$ is nonempty, that $ \left(  \cap_{i \in \I} \Fix T_{i} \right)$ is  closed and convex,    that $(\forall i \in \I)$ $T_{i}$ is Bregman isometric w.r.t.\,$f$, 
	and that $\overrightarrow{\CCO}_{\mathcal{S}}$ is at most single-valued on $\inte \dom f$. Let $y \in \inte \dom f$. Then the following statements hold.
	\begin{enumerate}
		\item \label{theorem:Convergence:Ef} $(\forall x \in   \dom f)$ $\inte \dom f \cap  \left(  \cap_{i \in \I} \Fix T_{i} \right)   \subseteq  \overrightarrow{E}_{f}(\mathcal{S}(x) ) $.
		\item  \label{theorem:Convergence:CCSEq}  $\left( \forall z \in \inte \dom f \cap  \left(  \cap_{i \in \I} \Fix T_{i} \right) \right)$ $\overrightarrow{\CCO}_{\mathcal{S}}y =\overleftarrow{\Pro}^{f}_{\aff(\mathcal{S}(y))}(z) \in \inte \dom f$.
		\item \label{theorem:Convergence:welldefined} $(\overrightarrow{\CCO}_{\mathcal{S}}^{k}y )_{k \in \mathbb{N}}$ is a well-defined sequence in $\inte \dom f$.
		\item \label{theorem:Convergence:BregMonot} $(\overrightarrow{\CCO}_{\mathcal{S}}^{k}y )_{k \in \mathbb{N}}$  is forward Bregman  monotone with respect to $     \cap_{i \in \I} \Fix T_{i}  $. Consequently, $	( \forall z \in \inte \dom f \cap  \left(  \cap_{i \in \I} \Fix T_{i} \right) )$ $ \left(  \D_{f} \left( \overrightarrow{\CCO}_{\mathcal{S}}^{k}y ,z \right)  \right)_{k \in \mathbb{N}}$ converges.  
		\item \label{theorem:Convergence:sequence} $\lim_{k \to \infty}$ $ \D_{f} \left(  \overrightarrow{\CCO}_{\mathcal{S}}^{k}y, \overrightarrow{\CCO}_{\mathcal{S}}^{k+1}y   \right) =0$.
		
		\item \label{theorem:Convergence:Browder} Suppose that   $(\forall i \in \I)$    $T_{i}$ is $f$-demiclosed,
  that for every sequence $( u_{k})_{k \in \mathbb{N}} $ in $ \inte \dom f$,  and that
	\begin{align} \label{eq:triangleineq}	 			 
	\begin{rcases}
	\D_{f} (u_{k}, \overrightarrow{\CCO}_{\mathcal{S}} u_{k} ) \to 0\\
	(\forall i \in \I) ~ \D_{f} (T_{i}u_{k}, \overrightarrow{\CCO}_{\mathcal{S}}u_{k} ) \to 0  
	\end{rcases}
	\Rightarrow  (\forall i \in \I) ~ \D_{f} (u_{k}, T_{i}u_{k} ) \to 0.
	\end{align}
	Then the following hold.
	\begin{enumerate}
		\item \label{theorem:Convergence:Browder:CCS}  $\overrightarrow{\CCO}_{\mathcal{S}}$ is $f$-demiclosed. 
		\item \label{theorem:Convergence:Browder:Cluster} Suppose that all weak sequential cluster   points of $(\overrightarrow{\CCO}_{\mathcal{S}}^{k}y )_{k \in \mathbb{N}}$ lie in $ \inte \dom f$. Then   $(\overrightarrow{\CCO}_{\mathcal{S}}^{k}y  )_{k \in \mathbb{N}}$ weakly converges to some point in $\inte \dom f \cap  \left(  \cap_{i \in \I} \Fix T_{i} \right) $.
	\end{enumerate}  
	
\end{enumerate}	
\end{theorem}

\begin{proof}
	Because  $(\forall i \in \I)$ $T_{i}$ is Bregman isometric, we know that $(\forall i \in \I)$ 
	$\left( \forall (x,y) \in \dom f \times \inte \dom f  \right)$ $\left( T_{i}x,T_{i}y \right) \in \dom f \times \inte \dom f $, and that
	\begin{align} \label{eq:Tisometry}
	(\forall i \in \I) ~\left(\forall \{x,y\}  \subseteq \dom f \times \inte \dom f \right) ~ \D_{f} \left( T_{i}x, T_{i} y   \right) = \D_{f}(x,y).
	\end{align} 
	
	\cref{theorem:Convergence:Ef}: Let  $x \in \dom f$ and $z \in \inte \dom f \cap  \left(  \cap_{i \in \I} \Fix T_{i} \right)$. Then
	\begin{align*}
	(\forall i \in \I) ~\D_{f} (x ,z) =\D_{f} (T_{i} x,z)   & \Leftrightarrow (\forall i \in \I) ~\D_{f} (x ,z) =\D_{f} (T_{i} x, T_{i}z),
	\end{align*}
	where the second equality holds by \cref{eq:Tisometry} and $z \in \cap_{j \in \I} \Fix T_{j}$.
	Hence,  $z \in \overrightarrow{E}_{f}( \mathcal{S}(x) )$ by the definition of $\overrightarrow{E}_{f}( \mathcal{S}(x) )$ presented in \cref{defn:CCS:Bregman:forward}.

	\cref{theorem:Convergence:CCSEq}: Let $z \in \inte \dom f \cap  \left(  \cap_{i \in \I} \Fix T_{i} \right) $. Because $y \in \inte \dom f$,   we see   that
	$y \in \aff \left( \mathcal{S}(y)  \right) \cap \inte \dom f \neq \varnothing$, and that, by 	\cref{theorem:Convergence:Ef} above,   $z \in     \overrightarrow{E}_{f}(\mathcal{S}(y) ) $. 
	Applying \cref{fact:CircumOperator}\cref{fact:CircumOperator:for} with $K$ replaced by $\mathcal{S}(y)$, we deduce that $ \overleftarrow{\Pro}^{f}_{\aff(\mathcal{S}(y))} z \in \overrightarrow{\CCO{}} (\mathcal{S}(y))$.
	Combine this with \cref{defn:BregCirMap},  \cref{fact:Pleft}, and the assumption that  $\overrightarrow{\CCO}_{\mathcal{S}}$ is at most single-valued,  to yield that  $\overrightarrow{\CCO}_{\mathcal{S}}y=	 \overrightarrow{\CCO{}} (\mathcal{S}(y)) =\overleftarrow{\Pro}^{f}_{\aff(\mathcal{S}(y))} z \in \inte \dom f$.

	\cref{theorem:Convergence:welldefined}$\&$\cref{theorem:Convergence:BregMonot}$\&$\cref{theorem:Convergence:sequence}: 
	By \cref{lemma:Fix}\cref{lemma:Fix:forward}, $ \Fix \overrightarrow{\CCO}_{\mathcal{S}} = \inte \dom f \cap  \left(  \cap_{i \in \I} \Fix T_{i} \right)$. Hence, \cref{theorem:Convergence:CCSEq} implies that	 	
	\begin{align*} 
	(\forall x \in   \inte \dom f )  ~( \forall z \in \Fix \overrightarrow{\CCO}_{\mathcal{S}} \cap \inte \dom f ) \quad \overrightarrow{\CCO}_{\mathcal{S}}x =\overleftarrow{\Pro}^{f}_{\aff(\mathcal{S}(x))}z \in \inte \dom f.
	\end{align*} 
	Therefore, the required results follow directly from \cref{theorem:conver:G}\cref{theorem:conver:G:sequence}$\&$\cref{theorem:conver:G:D}$\&$\cref{theorem:conver:G:limD}, respectively,  with $\mathcal{C} =\inte \dom f$ and $G= \overrightarrow{\CCO}_{\mathcal{S}}$.

	\cref{theorem:Convergence:Browder:CCS}:  Suppose that $x \in \inte \dom f$ and that $( x_{k} )_{k \in \mathbb{N}} $ is a sequence in $ \inte \dom f$ such that $x_{k} \weakly x$. Notice that, due to \cref{theorem:Convergence:CCSEq} above, 
$ \left( \overrightarrow{\CCO}_{\mathcal{S}} x_{k} \right)_{k \in \mathbb{N}} $  is a sequence in  $ \inte \dom f$. Assume that $\D_{f}(x_k, \overrightarrow{\CCO}_{\mathcal{S}}x_k) \to 0$.
In view of \cref{defn:BregCirMap},
\begin{align*} 
(\forall i \in \I) ~ (\forall k \in \mathbb{N}) \quad \D_{f} \left( x_{k} , \overrightarrow{\CCO}_{\mathcal{S}}x_k\right)   =\D_{f} \left( T_{i}x_{k} ,\overrightarrow{\CCO}_{\mathcal{S}}x_k\right),  
\end{align*}
which, connecting with $\D_{f}(x_k, \overrightarrow{\CCO}_{\mathcal{S}}x_k) \to 0$ and applying  \cref{eq:triangleineq} with $(\forall k \in \mathbb{N})$ $u_{k} =x_{k}$, implies that  $(\forall i \in \I)$  $ \D_{f} (x_{k}, T_{i}x_{k} ) \to 0 $.	
Let $i \in \I$.	Note that $T_{i}$ is $f$-demiclosed. So the results $x_{k } \weakly x$  and $ \D_{f} (x_{k}, T_{i}x_{k} ) \to 0 $ imply that $x \in \Fix T_{i}$. 
Because $i \in \I$ is chosen arbitrarily, 	by \cref{lemma:Fix}\cref{lemma:Fix:forward}, $x \in \inte \dom f \cap  \left(  \cap_{i \in \I} \Fix T_{i} \right) =\Fix \overrightarrow{\CCO}_{\mathcal{S}} $. 
	
	\cref{theorem:Convergence:Browder:Cluster}: Bearing  	\cref{theorem:Convergence:Browder:CCS} in mind, we observe that the desired convergence follows from 
	\cref{theorem:conver:G}\cref{theorem:conver:G:conv}   with $\mathcal{C} =\inte \dom f$ and $G= \overrightarrow{\CCO}_{\mathcal{S}}$.
\end{proof}

	\begin{remark} \label{remark:conditiontriangleinequality}
			We uphold assumptions in \cref{theorem:Convergence} and have a closer look at the condition \cref{eq:triangleineq}. As we mentioned before, in the classical Euclidean distance, the condition \cref{eq:triangleineq}  is immediate from  the triangle inequality of the Euclidean distance. Because generally Bregman distances do not obey the triangle inequality, \cref{eq:triangleineq} is no longer trivial if $f \neq \frac{1}{2} \norm{\cdot}$. We explain below that, even if the classical triangle inequality may not hold in a general Bregman distance, there is no need to be pessimistic.
			
	 	 Let $( u_{k})_{k \in \mathbb{N}} $ be in $ \inte \dom f$. According to \cref{defn:BregCirMap} and \cref{eq:SSx}, 
			\begin{align*}
			(\forall k \in \mathbb{N})	(\forall i \in \I) \quad		    \D_{f} (T_{i}u_{k}, \overrightarrow{\CCO}_{\mathcal{S}}u_{k} ) =\D_{f} (u_{k}, \overrightarrow{\CCO}_{\mathcal{S}} u_{k} ),
			\end{align*}
			which yields that 	\cref{eq:triangleineq} is equivalent to the following statement
			\begin{align*}  
			\D_{f} (u_{k}, \overrightarrow{\CCO}_{\mathcal{S}} u_{k} ) \to 0 \Rightarrow  (\forall i \in \I)  \D_{f} (u_{k}, T_{i}u_{k} ) \to 0.
			\end{align*}
			Combine this with  \cref{defn:BregmanDistance} to ensure that 	\cref{eq:triangleineq} holds if and only if  
			\begin{align*}
			&  f\left(u_{k}\right) -f\left(\overrightarrow{\CCO}_{\mathcal{S}} u_{k} \right) -\innp{\nabla f\left( \overrightarrow{\CCO}_{\mathcal{S}} u_{k}  \right), u_{k}-\overrightarrow{\CCO}_{\mathcal{S}} u_{k} } \to 0 \\
			\Rightarrow & (\forall i \in \I)  f\left(u_{k}\right) -f\left(T_{i}u_{k}\right) -\innp{\nabla f\left(T_{i}u_{k}\right), u_{k}-T_{i}u_{k}} \to 0.
			\end{align*}			
			Therefore, we observe that \cref{eq:triangleineq} depends on not only the function $f$ but also  the set $\mathcal{S}:= \{ \Id, T_{1}, \ldots, T_{m} \}$. For example,
			\begin{enumerate}
				\item  \label{remark:conditiontriangleinequality:Id} if $\mathcal{S} = \{ \Id \}$, then  \cref{eq:triangleineq} holds for an arbitrary function $f$  satisfying the statements of \cref{defn:BregmanDistance}; 
				\item \label{remark:conditiontriangleinequality:leq} if there exists a particular constant $\rho \in \mathbb{R}_{++}$ such that  $(\forall i \in \I) $ $(\forall k \in \mathbb{N})$ $\D_{f} (u_{k}, T_{i}u_{k} ) \leq \rho \D_{f} (u_{k}, \overrightarrow{\CCO}_{\mathcal{S}} u_{k} )$,   then \cref{eq:triangleineq} is also satisfied. 
			\end{enumerate}
		 	
	\end{remark}

 	Notice that if we can find a function $f \neq \frac{1}{2} \norm{\cdot}^{2}$  such that  $(\forall \{x,y,z\} \subseteq \mathcal{H})$ $D_{f}(x, y) \leq D_{f}(x, z) +D_{f}(z, y) $, then we obtain   \cref{eq:triangleineq} immediately. (We believe such functions actually have a   wide range of applications in various areas.)
 Based on our explanations in \cref{remark:conditiontriangleinequality}, even if we cannot find such beautiful functions, there is a large probability that there exist some special functions $f$ together with appropriate sets $\mathcal{S}:= \{ \Id, T_{1}, \ldots, T_{m} \}$ (the set might be dependent on the corresponding $f$) such that \cref{eq:triangleineq} is satisfied. 
 
 Given a function $f$ and a set $\mathcal{S}$, there may be infinitely many sequences $( u_{k})_{k \in \mathbb{N}} $ in $ \inte \dom f$ satisfying $ \D_{f} (u_{k}, \overrightarrow{\CCO}_{\mathcal{S}} u_{k} ) \to 0 $ and $
 (\forall i \in \I) $ $ \D_{f} (T_{i}u_{k}, \overrightarrow{\CCO}_{\mathcal{S}}u_{k} ) \to 0  $, so the workload of verifying \cref{eq:triangleineq}	 might be really large if we don't use tricks like \cref{remark:conditiontriangleinequality}\cref{remark:conditiontriangleinequality:leq} or some better ones. The following result demonstrates that to show the weak convergence of  forward Bregman circumcenter methods, 
the condition \cref{eq:triangleineq} might be unnecessary.

\begin{theorem} \label{theorem:xkConvergence}
Suppose that $f$ is Legendre,   that 	$\inte \dom f \cap  \left(  \cap_{i \in \I} \Fix T_{i} \right)$ is nonempty, that $ \left(  \cap_{i \in \I} \Fix T_{i} \right)$ is  closed and convex,  that $(\forall i \in \I)$ $T_{i}$ is Bregman isometric w.r.t.\,$f$, 
and that $\overrightarrow{\CCO}_{\mathcal{S}}$ is at most single-valued on $\inte \dom f$. Let $x_{0} \in \inte \dom f$. 
Set $(\forall k \in \mathbb{N})$ $x_{k+1} = \overrightarrow{\CCO}_{\mathcal{S}}x_{k}$.
Then the following assertions hold.
\begin{enumerate}
	\item   \label{theorem:xkConvergence:welldefined} $(x_{k})_{k \in \mathbb{N}}$ is a well-defined sequence in $\inte \dom f$.
	\item \label{theorem:xkConvergence:BregMonot} $(x_{k})_{k \in \mathbb{N}}$  is forward Bregman  monotone with respect to $     \cap_{i \in \I} \Fix T_{i}  $. Consequently, $	( \forall z \in \inte \dom f \cap  \left(  \cap_{i \in \I} \Fix T_{i} \right) )$ $ \left(  \D_{f} \left( x_{k} ,z \right)  \right)_{k \in \mathbb{N}}$ converges.  
	\item \label{theorem:xkConvergence:lim} $\lim_{k \to \infty}$ $ \D_{f} \left(  x_{k}, \overrightarrow{\CCO}_{\mathcal{S}}x_{k}  \right) =0$.
	
	\item \label{theorem:xkConvergence:weakconver} Suppose that   $(\forall i \in \I)$    $T_{i}$ is $f$-demiclosed,  that  all weak sequential cluster   points of $(x_{k})_{k \in \mathbb{N}}$ lie in $ \inte \dom f$, and that 
	\begin{align} \label{eq:triangleineq:update}
	\D_{f} \left(x_{k}, \overrightarrow{\CCO}_{\mathcal{S}} x_{k} \right) \to 0 \Rightarrow  (\forall i \in \I)  \D_{f} (x_{k}, T_{i}x_{k} ) \to 0.
	\end{align}
	Then   $(x_{k})_{k \in \mathbb{N}}$ weakly converges to some point in $\inte \dom f \cap  \left(  \cap_{i \in \I} \Fix T_{i} \right) $.
\end{enumerate}
\end{theorem} 

\begin{proof}
\cref{theorem:xkConvergence:welldefined}$\&$\cref{theorem:xkConvergence:BregMonot}$\&$\cref{theorem:xkConvergence:lim}: These results are clear from \cref{theorem:Convergence}\cref{theorem:Convergence:welldefined}$\&$\cref{theorem:Convergence:BregMonot}$\&$\cref{theorem:Convergence:sequence}, respectively.

\cref{theorem:xkConvergence:weakconver}: According to \cref{theorem:xkConvergence:BregMonot}  above and \cref{theorem:ConvBregBackwMonot}\cref{theorem:ConvBregBackwMonot:conv},   it suffices to prove that all weak sequential cluster points of $(x_{k})_{k \in \mathbb{N}}$ are in $  \cap_{i \in \I} \Fix T_{i}   $.

Suppose that a subsequence $(x_{k_{j}})_{j \in \mathbb{N}}$  of $(x_{k})_{k \in \mathbb{N}}$ weakly converges to $\bar{x} \in  \inte \dom f$.  Taking \cref{theorem:xkConvergence:lim} above and \cref{eq:triangleineq:update} into account, we observe that $(\forall i \in \I) $ $ \D_{f} (x_{k}, T_{i}x_{k} ) \to 0$, which ensures that $(\forall i \in \I) $ $ \D_{f} (x_{k_{j}}, T_{i}x_{k_{j}} ) \to 0$. Now combine the results,  $x_{k_{j}} \weakly \bar{x}$ and $(\forall i \in \I) $ $ \D_{f} (x_{k_{j}}, T_{i}x_{k_{j}} ) \to 0$, with the assumption  that   $(\forall i \in \I)$    $T_{i}$ is $f$-demiclosed to necessitate that $\bar{x} \in    \cap_{i \in \I} \Fix T_{i} $. 

Altogether, the proof is done. 
\end{proof}

\begin{remark}
	\begin{enumerate}
		\item Note that  \cref{eq:triangleineq} is clearly stronger than \cref{eq:triangleineq:update} and that  \cref{eq:triangleineq} is implied by the triangle inequality of the Euclidean distance. So \cref{eq:triangleineq:update} is trivial in the Euclidean distance as well. 
 	\item  \label{remark:conditiontriangleinequality:finiteconvergence} Suppose that there exists $K \in \mathbb{N}$ such that $x_{K} \in \cap_{i \in \I} \Fix T_{i}$. Then, via \cref{lemma:Fix}\cref{lemma:Fix:forward},  
		\begin{align*}
		x_{K} \in  \inte \dom f \cap  \left(  \cap_{i \in \I} \Fix T_{i} \right) = \Fix \overrightarrow{\CCO}_{\mathcal{S}}.
		\end{align*}
		Hence, 
		\begin{align*}
		(\forall i \in \I) (\forall k \geq K) \quad \D_{f} (x_{k}, \overrightarrow{\CCO}_{\mathcal{S}} x_{k} ) = \D_{f} (T_{i}x_{k}, \overrightarrow{\CCO}_{\mathcal{S}}x_{k} )=  \D_{f} (x_{k}, T_{i}x_{k} ) =0,
		\end{align*}
		which guarantees that if the forward Bregman circumcenter method converges to a point in $\inte \dom f \cap  \left(  \cap_{i \in \I} \Fix T_{i} \right) $ in finitely many steps, then   \cref{eq:triangleineq} holds.
		
			Note that the finite convergence of the circumcenter method under the Euclidean distance is not a big surprise (for details, see e.g., \cite{BBCS2020ConvexFeasibility}, \cite{BBCS2020CRMbetter} and \cite{Ouyang2020Finite}). Moreover, the following \cref{example:CCS} also shows the  one-step convergence of a forward Bregman circumcenter method under a general Bregman distance.   
	\end{enumerate}
\end{remark}

We leave a systematic study on the conditions  \cref{eq:triangleineq} and \cref{eq:triangleineq:update} as future work.

To end this work, we revisit the operator $T$ and the function $f$ in  \cref{prop:FDE}, which satisfy all  requirements in \cref{theorem:xkConvergence}.
In view of \cref{prop:FDE}, the following example illustrates \Cref{theorem:Convergence,theorem:xkConvergence} and  demonstrates that the forward Bregman circumcenter method finds the desired fixed point with one iterate.   
\begin{example} \label{example:CCS}
	Suppose that $\mathcal{H} =\mathbb{R}^{n}$. Define $f :x \mapsto   \sum^{n}_{i=1} x_{i} \ln  (x_{i}) + (1-x_{i} ) \ln(1-x_{i} )$, with $\dom f =\left[0, 1\right]^{n} $, which is the Fermi-Dirac entropy. 
	Let $\Lambda$ be a subset of $\J := \{ 1, \ldots,n  \}$. 
	Define $T: \mathbb{R}^{n} \to \mathbb{R}^{n}$ by $\left( \forall x  \in \mathbb{R}^{n} \right)$ $T x := (y_{i})^{n}_{i=1}$ where  $(\forall i \in \J \smallsetminus \Lambda )$ $ y_{i} =x_{i}$  and $(\forall i \in   \Lambda )$ $y_{i} = 1-x_{i}$.  Set $\mathcal{S} := \{ \Id , T\} $.
	 Let $x := (x_{i})_{i \in \J} \in \left[0, 1\right]^{n}$.  Denote by $\Phi := \{ i\in \J ~:~ x_{i} \neq \frac{1}{2} \}$.
	Then the following statements hold. 
\begin{enumerate}
	\item  \label{example:CCS:Ef} $\overrightarrow{E}_{f}( \mathcal{S}(x)) = \left\{ (p_{i})_{i \in \J}  \in \left]0, 1\right[^{n} ~:~ 0=\sum_{i \in \Lambda} (2x_{i} -1) \ln \left( \frac{ p_{i}}{ 1-p_{i}} \right)  \right\}$.
	\item \label{example:CCS:p} $	\overrightarrow{\CCO}_{\mathcal{S}} (x) = \left\{ (p_{i})_{i \in \J}  \in \Big( \left]0, 1\right[^{n} \cap \aff \{ x, 1-x\}  \Big) ~:~ 0=\sum_{i \in \Lambda}  (2x_{i} -1) \ln \left( \frac{ p_{i}}{ 1-p_{i}} \right)  \right\}$. Moreover,  
	\begin{align*}
		& \left\{ (p_{i})_{i \in \J}  \in \Big( \left]0, 1\right[^{n} \cap \aff \{ x, 1-x\}  \Big) ~:~  (\forall i \in \Lambda \smallsetminus \Phi) p_{i} =\frac{1}{2} \right\} \subseteq \overrightarrow{\CCO}_{\mathcal{S}} (x);\\
&\varnothing \neq  \left\{ (p_{i})_{i \in \J}  \in \Big( \left]0, 1\right[^{n} \cap \aff \{ x, 1-x\}  \Big) ~:~ (\forall i \in \Lambda) p_{i} =\frac{1}{2} \right\} \subseteq	\Fix T \cap \inte \dom f  \cap 	\overrightarrow{\CCO}_{\mathcal{S}} (x).
	\end{align*}
	\item \label{example:CCS:Phi} Suppose that $(\forall i \in \Lambda \smallsetminus \Phi)$ $x_{i} > \frac{1}{2}$ or  $(\forall i \in \Lambda \smallsetminus \Phi)$ $x_{i}  < \frac{1}{2}$. Then 
		\begin{align*}
 \overrightarrow{\CCO}_{\mathcal{S}} (x) = \left\{ (p_{i})_{i \in \J}  \in \Big( \left]0, 1\right[^{n} \cap \aff \{ x, 1-x\}  \Big) ~:~  (\forall  i \in \Lambda \smallsetminus \Phi) p_{i} =\frac{1}{2}   \right\}.
	\end{align*}
		\item \label{example:CCS:PhiFix}Suppose that $(\forall i \in \Lambda \smallsetminus \Phi)$ $x_{i} > \frac{1}{2}$ or  $(\forall i \in \Lambda \smallsetminus \Phi)$ $x_{i}  < \frac{1}{2}$, and that $\Lambda \cap \Phi = \varnothing$. Then 
			\begin{align*}
 	\overrightarrow{\CCO}_{\mathcal{S}} (x) = \left\{ (p_{i})_{i \in \J}  \in \Big( \left]0, 1\right[^{n} \cap \aff \{ x, 1-x\}  \Big) ~:~  (\forall  i \in \Lambda ) p_{i} =\frac{1}{2}   \right\} \subseteq \Fix T \cap \inte \dom f.
		\end{align*}
		\item   \label{example:CCS:singleton} Suppose that $\Lambda $ is a singleton, say $\Lambda := \{i_{0}\}$. Then if $x_{i_{0}} =\frac{1}{2}$, then $\overrightarrow{\CCO}_{\mathcal{S}} (x) =  \left]0, 1\right[^{n} \cap \aff \{ x, 1-x\}  $; if $x_{i_{0}} \neq \frac{1}{2}$, then 
	\begin{align*}
 	\overrightarrow{\CCO}_{\mathcal{S}} (x) = \left\{ (p_{i})_{i \in \J}  \in  \left]0, 1\right[^{n} \cap \aff \{ x, 1-x\}   ~:~ p_{i_{0}} =  \frac{1}{2}  \right\}  \subseteq   \Fix T \cap \inte \dom f.
	\end{align*}
	\item  \label{example:CCS:Ef:R} Suppose that $\mathcal{H} =\mathbb{R}$. Then $	\left(\forall x  \in  \left[0, 1\right]  \right) $ $ \overrightarrow{\CCO}_{\mathcal{S}}(x) = \left\{ \frac{1}{2} \right\} = \inte \dom f \cap \Fix T$.
\item \label{example:CCS:Ef:R2} Suppose that $\mathcal{H} =\mathbb{R}^{2}$. Let $x=(\frac{1}{4}, \frac{5}{6})$. Then $ \overrightarrow{\CCO}_{\mathcal{S}}(x) = \left\{ \frac{1}{2} \right\} = \inte \dom f \cap \Fix T$.
\end{enumerate}	
\end{example}

 \begin{proof}
 	\cref{example:CCS:Ef}: This is easy from the related definitions and \cref{eq:prop:FDE} in the proof of \cref{prop:FDE}.
 	
 	\cref{example:CCS:p}: Based on \cref{prop:FDE}\cref{prop:FDE:Ti},
 	\begin{align*}
 	\Fix T \cap \inte \dom f = \left\{ x \in  \left]0, 1\right[^{n}~:~  (\forall i \in \Lambda)   x_{i} = \frac{1}{2}  \right\}.
 	\end{align*}
 	Hence, this is clear from \cref{defn:BregCirMap}.
 	
  \cref{example:CCS:Phi}: This follows immediately from \cref{example:CCS:p}.
 	
 \cref{example:CCS:PhiFix}$\&$\cref{example:CCS:singleton}$\&$\cref{example:CCS:Ef:R}: 
 The required results follows easily from \cref{example:CCS:Phi}.
 
 \cref{example:CCS:Ef:R2}: This is from \cref{example:CCS:p}  and some basic calculus. 
 \end{proof}

  Searching more particular  examples of   sets $\mathcal{S}= \{\Id, T_{1}, \ldots, T_{m} \} $  satisfying the requirements in \Cref{theorem:Convergence,theorem:xkConvergence}   under general Bregman distances associated with $f \neq \frac{1}{2} \norm{\cdot}^{2}$  is left as our  future work.

	\section*{Acknowledgements}
The author sincerely appreciates the useful comments from the associate editor and two anonymous referees, which were very helpful in improving the exactness and presentation of this paper. The author also would like to thank Dr.\,Xianfu Wang, who carefully read the original draft of this work and provided kind comments for the improvements.

\addcontentsline{toc}{section}{References}

\bibliographystyle{abbrv}

\end{document}